\documentclass[12pt]{article}
\usepackage{pdfsync}
\usepackage{lscape}
\usepackage{longtable}
\usepackage{amsmath}
\usepackage{amsfonts}
\usepackage{amssymb}
\usepackage{amsmath,amscd}
\usepackage{amsthm}
\usepackage[usenames,dvipsnames]{color}
\usepackage[dvips]{graphicx}
\usepackage{hyperref}
\usepackage{enumerate}
\usepackage{marginnote}

\newcommand{\al}{\alpha}

\newcommand{\la}{\lambda}

\newcommand{\A}{\mathbb A}
\newcommand{\C}{\mathbb C}
\newcommand{\FF}{\mathbb F}
\newcommand{\PP}{\mathbb P}

\renewcommand{\P}{\mathbb P}
\newcommand{\ZZ}{\mathbb Z}
\newcommand{\F}{\mathbb F}
\newcommand{\Q}{\mathbb Q}
\newcommand{\Z}{\mathbb Z}

\newcommand{\D}{\mathbb D}
\newcommand{\cO}{\mathcal O}

\newcommand{\cB}{\mathcal B}

\newcommand{\cD}{\mathcal D}

\newcommand{\cM}{\mathcal M}

\newcommand{\cY}{\mathcal Y}
\newcommand{\cV}{\mathcal V}
\newcommand{\cU}{\mathcal U}
\newcommand{\Xtilde}{\widetilde X}
\newcommand{\Ytilde}{\widetilde Y}
\newcommand{\Ztilde}{\widetilde Z}
\newcommand{\Vtilde}{\widetilde V}
\newcommand{\Sthree}{\mathfrak{S}_3}

\newcommand{\eps}{\varepsilon}

\DeclareMathOperator{\Fix}{Fix}

\DeclareMathOperator{\Proj}{Proj}

\newtheorem{thm}{Theorem}
\newtheorem{cor}[thm]{Corollary}
\newtheorem{lem}[thm]{Lemma}
\newtheorem{prop}[thm]{Proposition}

\theoremstyle{definition}

\theoremstyle{remark}
\newtheorem{rem}[thm]{Remark}
\theoremstyle{remark}

\numberwithin{equation}{section}
\numberwithin{thm}{section}
\title{On $\ZZ/3$-Godeaux surfaces}
\author{Stephen Coughlan \and Giancarlo Urz\'ua}

\newcommand{\Addresses}{{% additional braces for segregating \footnotesize
  \bigskip
  \footnotesize

  S.~Coughlan, \textsc{Institut f\"ur Algebraische Geometrie, Leibniz Universit\"at Hannover, Welfengarten 1, D-30167 Hannover, Germany.}

\textit{Current address:} \textsc{Mathematisches Institut, Lehrstuhl Mathematik VIII, Universit\"atsstrasse 30, D-95447 Bayreuth, Germany.} \par\nopagebreak
  \par\nopagebreak
  \textit{E-mail address:} \texttt{stephencoughlan21@googlemail.com}

  \medskip

  G.~Urz\'ua, \textsc{ Facultad de Matem\'aticas, Pontificia Universidad Cat\'olica de Chile, Campus San Joaqu\'in, Avenida Vicu\~na Mackenna 4860, Santiago, Chile.}

  \textit{Current address:} \textsc{Department of Mathematics and Statistics, University of Massachusetts, 710 N. Pleasant Street, Amherst, MA 01003-9305, USA.} \par\nopagebreak
  \textit{E-mail address:} \texttt{gianurzua@gmail.com}

}}

\date{\today}

\begin{document}
\maketitle

\begin{abstract}
We prove that Godeaux--Reid surfaces with torsion group $\ZZ/3$ have topological fundamental group $\ZZ/3$. For this purpose, we describe degenerations to stable KSBA surfaces with one $\frac14(1,1)$ singularity, whose minimal resolution are elliptic fibrations with two multiplicity $3$ fibres and one $I_4$ singular fibre. We study special such degenerations which have an involution, describing the corresponding Campedelli double plane construction. We also find some stable rational degenerations, some of which have more singularities, and one of which has a single $\frac19(1,2)$ singularity, the minimal possible index for such a surface. Finally, we do the analogous study for the Godeaux surfaces with torsion $\ZZ/4$.
\end{abstract}
\setcounter{tocdepth}{1}
\tableofcontents

%------------------------------------------------------------------------------------------------------
\section{Introduction}
Godeaux surfaces are surfaces of general type with $p_g=q=0$ and $K^2=1$. Such surfaces have been a classical object of study for a long time. We refer to \cite{BCP} and \cite{Dolg} for surveys of differing vintage. In \cite{Miy}, Miyaoka showed among other things, that Godeaux surfaces have cyclic torsion group $H^2(X,\ZZ)_{\text{tors}}$ of order at most five, and Reid \cite{R} gave explicit constructions of the surfaces and their moduli spaces when the torsion has order $3,4$ or $5$. In the two latter cases, it is clear from Reid's construction that the topological fundamental group is isomorphic to the torsion group, but the fundamental group in the $\Z/3$ case has not been computed before. In this article, we refer to such surfaces as Godeaux--Reid surfaces, (or simply $\ZZ/3$-Godeaux surfaces). Our main result is that the topological fundamental group of the Godeaux--Reid surface is indeed $\Z/3$.

We consider KSBA stable degenerations \cite{KSB} of certain Godeaux--Reid surfaces with an involution. Specifically, we study a boundary divisor corresponding to Godeaux surfaces with a single $\frac14(1,1)$ singularity. Those with an involution have a resolution of singularities as described in the abstract above, and we can compute the topological fundamental group explicitly.

\begin{thm} The topological fundamental group of the Godeaux--Reid surface is $\Z/3$.
\end{thm}

The presence of an involution, which was only recently discovered by Reid \cite{God3}, is crucial in the configuration of the singular fibres of the elliptic fibration, which in turn is important for the computation of the fundamental group. We also analyse the quotient, describing the Godeaux surface as a Campedelli double plane, see \cite{CCML} and \cite{KL}.

\begin{thm} The Godeaux--Reid surface with an involution is a Campedelli double plane. That is, the quotient gives a double cover of $\PP^2$ branched in a curve of degree 10 with five $(3,3)$ points and one 4-point.
\end{thm}

When the Godeaux surface is stable with a single $\frac14(1,1)$ singularity and an involution, the branch curve breaks into two quintics, now with four $(3,3)$-points, one $4$-point and a worse singularity of the form $\{(y^2-x^2)(y^2-x^4)=0\}$ in a neighbourhood of $0\in\C^2$; see Proposition \ref{campe}. This splitting of the branch curve also occurs in degenerations of the Craighero--Gattazzo surface considered in \cite{RTU}.

We also find the quotient by this involution of the corresponding family of universal coverings, obtaining after a flip and a divisorial contraction a special family of K3 surfaces with nodes; see Subsection \ref{flipped}.

We then consider stable $\Z/3$-Godeaux surfaces with other types of Wahl singularity, through the phenomenon described by Kawamata in \cite{K} of confluence of a multiple fibre and a non-multiple singular fibre in an elliptic fibration. We show (see Subsection \ref{s23})

\begin{thm}
There are stable rational $\Z/3$-Godeaux surfaces with a single $\frac19(1,2)$ singularity, with one $\frac14(1,1)$ and one $\frac19(1,2)$ singularities, and with one $\frac14(1,1)$ and two $\frac19(1,2)$ singularities. The latter has a model as a blown up plane from a pencil of cubics.
\end{thm}

The minimal resolution of a stable Godeaux surface with a $\frac14(1,1)$ must have Kodaira dimension $1$ (see e.g.~\cite[Remark 5.3]{Urz}), and so the smallest index of a Wahl singularity in a rational stable Godeaux surface is achieved by $\frac19(1,2)$. We remark that for simply connected Godeaux surfaces, the singularity $\frac19(1,2)$ appears in rational and non-rational stable surfaces; see \cite[Table 1]{SU}.

All stable degenerations mentioned above, have completely explicit descriptions in terms of Reid's original moduli parameters. In particular, we exhibit large subvarieties and even divisors in the boundary of the KSBA moduli space, and further degenerations can be studied simply by specialising the parameters.

In the last section, we make the corresponding study for $\Z/4$-Godeaux surfaces, where the covering surface is a complete intersection and so it is easier to understand. This also serves as a key to unlock the more complicated details of the $\ZZ/3$ case. We know of two different involutions on such surfaces, and we consider the one whose quotient is an Enriques surface \cite{KL}, \cite{MLP}. We get similar results concerning stable degenerations of the double cover, including the phenomenon of the branch curve breaking into two pieces. We refer to Section \ref{s3} for details.

Gorenstein stable degenerations of Godeaux surfaces have been studied recently in a series of articles \cite{FPR1}, \cite{FPR2}. Our work uses stable Godeaux surfaces of Gorenstein index two and three.

\subsection*{Acknowledgements}
The first author is partially supported by the DFG through grant number Hu 337-6/2, and the second author is supported by the FONDECYT regular grant 1150068 funded by the Chilean Government.

%------------------------------------------------------------------------------------------------------
\section{Reid's construction of $\ZZ/3$-Godeaux surfaces} \label{s1}
In \cite{R}, Reid showed that the moduli space of $\ZZ/3$-Godeaux surfaces is irreducible and 8-dimensional, by constructing surfaces $Y$ of general type with $p_g=2$, $K^2=3$ and with fixed point free $\ZZ/3$-action $\sigma$, such that the quotient $Y/\sigma=X$ is the Godeaux surface. More recently, Reid \cite[\S7]{Kinosaki}, \cite{God3} clarified the construction of $Y$ using a parallel unprojection and a 4-dimensional key variety. We briefly describe this construction below.

%-------------------------------------------------------------------------------------------------------
\subsection{The key variety $W\subset\PP(1^3,2^3,3^3)$}
Let $W\subset\PP(1^3,2^3,3^3)$ be the 4-dimensional $\ZZ/3$-invariant variety defined by 9 equations $R_0,R_1,R_2$, $S_0,S_1,S_2$, $T_0,T_1,T_2$ formed by taking the orbits of
\[\renewcommand{\arraystretch}{1.2}
\begin{array}{lr}
R_0\colon & -x_0z_0 + y_1y_2 - r_0x_1x_2 = 0 \\
%-x1*z1 + y0*y2 - r1*x0*x2,
%-x2*z2 + y0*y1 - r2*x0*x1,
S_0\colon & -y_0z_0 + r_1x_2y_1 + r_2x_1y_2 + sx_1x_2 = 0 \\
%-y1*z1 + r0*x2*y0 + r2*x0*y2 + S*x0*x2,
%-y2*z2 + r0*x1*y0 + r1*x0*y1 + S*x0*x1,
T_0\colon & -z_1z_2 + r_0y_0^2 + sx_0y_0 + r_1r_2x_0^2 = 0
%-z0*z2 + r1*y1^2 + S*x1*y1 + r0*r2*x1^2,
%-z0*z1 + r2*y2^2 + S*x2*y2 + r0*r1*x2^2
\end{array}\]
under the $\ZZ/3$-action $\sigma$. The coordinates $x_i$, $y_i$, $z_i$ for $i=0,1,2$ have respective weights $1,2,3$, and $\sigma$ acts on them by cyclic permutation of indices $(012)$:
\[\sigma\colon x_i\mapsto x_{\sigma(i)},\ y_i\mapsto y_{\sigma(i)},\ z_i\mapsto z_{\sigma(i)}.\]
Moreover, $r_i$ (respectively $s$) are weighted homogeneous of degree 2 (resp.~$3$) and chosen so that $\sigma$ permutes the indices of $r_i$ (resp.~fixes $s$). Let $A$ denote the restriction of the ample generator of the ambient space to $W$, that is, $\cO_W(A)=\cO_W(1)$. Then according to \cite{God3}, $W$ is a Fano 4-fold with $K_W=-3A$, $A^3=1$ and $3\times\frac13(1,1,2,2)$ points.

\begin{thm}[Reid \cite{R,God3}] The canonical model of the covering surface $Y' \subset\PP(1^2,2^3,3^2)$ is the intersection of $W$ with the $\sigma$-invariant weighted linear subspace
\[Y'=W\cap(x_0+x_1+x_2=z_0+z_1+z_2=0).\]
\end{thm}
Reid proves that for general choices of $r_i$ and $s$, the surface $Y'$ is smooth and irreducible (this is checked by computer), and the action of $\sigma$ on $Y'$ is fixed point free. Thus the quotient surface $X'$ is a Godeaux surface with torsion $\ZZ/3$.

\begin{prop}[Reid~{\cite[\S3]{R}}]\label{prop!moduli} The coarse moduli space of $\ZZ/3$-Godeaux surfaces is irreducible, unirational of dimension 8, covered by the 9-dimensional parameter space given by the following forms for $r_i$ and $s$
\begin{gather*}
r_0=a_{11}x_1^2+a_{12}x_1x_2+a_{22}x_2^2+b_0y_0+b_1y_1,\ r_1=\sigma(r_0),\ r_2=\sigma(r_1), \\
s=c_2(x_0^2x_1+x_1^2x_2+x_2^2x_0)+c_3(x_0^2x_2+x_1^2x_0+x_2^2x_1)\\
+d_2(x_0y_1+x_1y_2+x_2y_0)+d_3(x_0y_2+x_1y_0+x_2y_1).
\end{gather*}

\end{prop}

%{\tiny \marginnote{G: I still do not get the $8$ parameters. Please elaborate :)}}

\begin{proof} This is just a translation of Reid's result \cite{R} to the new key variety description of $Y'$. It is easy to write out the general forms for $r_i$ and $s$:
\begin{gather*}
r_0=\sum_{i\le j}a_{ij}x_ix_j+\sum_ib_iy_i,\ r_1=\sigma(r_0),\ r_2=\sigma(r_1), \\
s=c_1\Sigma(x_0^3)+c_2\Sigma(x_0^2x_1)+c_3\Sigma(x_0^2x_2)+c_4x_0x_1x_2
+d_1\Sigma(x_0y_0)+d_2\Sigma(x_0y_1)+d_3\Sigma(x_0y_2),
\end{gather*}
where $\Sigma(m)$ denotes the orbit sum $m+\sigma(m)+\sigma^2(m)$ of $m$ under the action of $\ZZ/3$.

On the other hand, the terms of $r_i$ involving $x_0$ are redundant. Moreover, we can assume $b_2=0$ using coordinate changes of the form $y_i\mapsto y_i-\varepsilon x_jx_k$. For $s$, there are four independent sections of the invariant eigenspace $H^0(3K_{Y'})^\text{inv}$ and we just choose them in a nice way.

Finally, a coordinate transformation of the form $x_i\mapsto\lambda x_i$, $y_i\mapsto y_i$, $z_i\mapsto\lambda^{-1}z_i$ for parameter $\lambda$ reduces us to an 8-dimensional moduli space, because the transformation acts by scaling the parameters in $r_i$ and $s$.
\end{proof}

We write $\cM$ for the KSBA moduli space of $\ZZ/3$-Godeaux surfaces \cite{KSB}. We do not distinguish between $\cM$ and the moduli space of covering surfaces.

Recall from \cite{R} that the blowup $\phi_{Y'}\colon\Ytilde'\to Y'$ of the three basepoints of $|K_{Y'}|$ has a model inside the scroll $\FF=\Proj_{\PP^1}(\cO\oplus 3\cO(-2))$ with base and fibre coordinates $(x_1,x_2)$ and $(t,y_0,y_1,y_2)$ respectively. Ignoring the $t$ variable for simplicity (see proof of Lemma \ref{lem!Z3-curve-section} below and \cite[\S3]{R}), the five defining equations of $\Ytilde'$ are
\begin{align*}
f &\colon x_1x_2R_0+x_2x_0R_1+x_0x_1R_2 =0 \\
g &\colon x_0S_0 - y_0R_1\equiv x_1S_1 - y_1R_2 \equiv x_2S_2 - y_2R_0=0 \\
h_0 &\colon x_1x_2S_0+y_0x_2R_1+y_0x_1R_2=0
\end{align*}
and $h_1=\sigma(h_0)$, $h_2=\sigma(h_1)$.
%Here $R'_i=R_i(\sqrt t x_j,y_k)$, $S'_i=S_i(x_j,y_k)$ are rehomogenisations of $R_i$, $S_i$, where in the definition of $R'_i$ and $S'_i$ we replace $r_i$ by $r'_i=r_i(\sqrt t x_j,y_k)$ and $s$ by $s'=s(t^\frac23 x_j,t^\frac13 y_k)$ (see forthcoming Lemma).
In fact, the general fibre of $\Ytilde'\to\PP^1$ is a weighted complete intersection $f_4=g_6=0$ in $\PP(1,2,2,2)$, and $h_0$ (respectively $h_1,h_2$) is used only to cut out the fibre over the distinguished point $P_0=(0,1)$, (resp.~$P_1=(1,1)$, $P_2=(0,1)$) of the base $\P^1$.

\begin{lem}\label{lem!Z3-curve-section} Let $Y'$ be a surface in $\cM$ and suppose $C$ is a curve in $|K_{Y'}|$ corresponding to the section $(\al_0,\al_1,\al_2)$ of $H^0(K_{Y'})$, where $\al_0+\al_1+\al_2=0$. Then $C$ is isomorphic to the following complete intersection of a quadric and a cubic in $\PP^3$ with coordinates $t^2,y_0,y_1,y_2$:
\begin{enumerate}
\item  If the $\al_i$ are all nonzero,
\begin{gather*}
\al_1\al_2(y_1y_2-r_0x_1x_2) + \al_0\al_2(y_0y_2-r_1x_0x_2) + \al_0\al_1(y_0y_1-r_2x_0x_1)=0, \\
-y_0y_1y_2 + r_0x_1x_2y_0 + r_1x_0x_2y_1 + r_2x_0x_1y_2 + sx_0x_1x_2=0
\end{gather*}
where $x_i$, $r_i$ and $s$ are evaluated at $x_i=\al_it$.
\item If $\al_0$ vanishes,
\begin{gather*}
y_1y_2-r_0x_1x_2 = 0,\\
y_0^2(y_1-y_2)-x_1(r_1x_2y_1+r_2x_1y_2+sx_1x_2) = 0
\end{gather*}
where now $x_i$, $r_i$ and $s$ are evaluated at $x_0=0$, $x_1=t$, $x_2=-t$.
\end{enumerate}
Curves corresponding to $\al_1=0$ and $\al_2=0$ are obtained in a similar manner, or from the action of $\sigma$ on $C_{\al_0}$.
\end{lem}
\begin{proof} Substitute $x_i=\al_it$ for $i=0,1,2$ in equations $f,g$. For reasons of bihomogeneity, we can divide $f$ by $t^2$, and this gives the two relations in part (i). The $h_i$ are redundant here because of the syzygies
\[
y_0f+x_1x_2g\equiv x_0h_0,\ y_1f+x_0x_2g\equiv x_1h_1,\ y_2f+x_0x_1g\equiv x_2h_2.
\]
For part (ii), examining the same syzygies, we see that if $\al_0$ vanishes, then we must take $f$ and $h_0$, because now $g,h_1,h_2$ are redundant.
\end{proof}

%We eliminate variables $z_0$, $z_1$, $z_2$ from equations by taking the following combinations: $\al_1\al_2R_0+\al_0\al_2R_1+\al_0\al_1R_2$ and any one of $y_iR_i-x_iS_i$ for some $i=0,1,2$. The first involves the $z_i$ only in the form $\al_0\al_1\al_2(z_0+z_1+z_2)$ which is identically zero on $Y$, and the second combination eliminates $z_i$ from $R_i$ and $S_i$, giving the same equation regardless of which $i$ is chosen.

%If $\al_0$ vanishes, then the above combinations do not work because we do not get a multiple of $z_0+z_1+z_2$. Instead, we take $R_0$ and $y_0(R_2-R_1)-x_1S_0$. The first equation no longer involves $z_0$ because $\al_0=0$, and the second eliminates $z_1$ and $z_2$ because $\al_1+\al_2=0$ when $\al_0=0$. Similar adjustments work if $\al_1=0$ or $\al_2=0$.

%\end{proof}

%\begin{rem} In what follows we only use part (ii) of the above lemma.
%\end{rem}

\begin{rem} The automorphism $\sigma$ acts on the base of $Y'\to\PP^1$ and there are two $\sigma$-invariant curves $F_\omega$ and $F_{\omega^2}$ in $|K_{Y'}|$ corresponding to $(1,\omega,\omega^2)$ and $(1,\omega^2,\omega)$ respectively, where $\omega$ is a primitive cube root of unity. The action of $\sigma$ on these curves is $(t,y_0,y_1,y_2)\mapsto(\omega^2t,y_1,y_2,y_0)$. If we further assume that $Y'$ has $\Sthree$ symmetry (see Section \ref{sec!Sthree} below), then $F_\omega$ has equations
\begin{equation}\label{eq!multiple-fibre}
\begin{gathered}
y_1y_2+\omega^2 y_0y_2+\omega y_0y_1 - b_0t(y_0+\omega y_1 + \omega^2 y_2)  + 3(a_{11}-a_{12})t^2 =0 \\
-y_0y_1y_2+b_0t(y_0^2+\omega^2y_1^2+\omega y_2^2)
+(a_{12}-a_{11}-d_2)t^2(y_0+\omega y_1 +\omega^2 y_2)-3c_2t^3=0
\end{gathered}
\end{equation}
and $F_{\omega^2}$ is similar, with $\omega$ and $\omega^2$ interchanged.
\end{rem}

%------------------------------------------------------------------------------------------
\subsection{A family with $\Sthree$ symmetry and quotients}\label{sec!Sthree}

%{\tiny \marginnote{G: I am changing notation everywhere, please check.}}

As observed in \cite{God3}, there is a special subfamily $\cM^s\subset\cM$ for good choices of $r_i$ and $s$, such that the general surface $Y'$ in $\cM^s$ has a larger automorphism group $\Sthree$ rather than just $\ZZ/3$. This family is defined by
\[\cM^s\colon (a_{11}=a_{22},\ b_1=0,\ c_2=c_3,\ d_2=d_3)\subset\cM.\]
and $\Sthree$ acts by permutation on the indices of $x_i,y_i,z_i$. The family $\cM^s$ is irreducible and 4-dimensional. The action of $\Sthree$ on $Y'$ descends to an involution on the $\ZZ/3$-Godeaux surface.

\begin{rem} From now on, we will use the superscript $s$ to denote the intersection $\cD^s=\cD\cap\cM^s$ of a stratum $\cD\subset\cM$ with the $\Sthree$-symmetric family $\cM^s$.
\end{rem}

Let $\tau_0=(12)$, $\tau_1=(02)$, $\tau_2=(01)$ be the involutions in $\Sthree$ acting on $Y'$ in $\cM^s$. They are conjugate under the action of $\sigma$, so we use $\tau=\tau_0$ for our computations. The action of $\tau$ on $\PP(1^3,2^3,3^3)$ has eigenspace decomposition

\[\begin{array}{ccc}
\text{Degree} & + & - \\
\hline
1 & x_0,\ x_1+x_2 & x_1-x_2 \\
2 & y_0,\ y_1+y_2 & y_1-y_2 \\
3 & z_0,\ z_1+z_2 & z_1-z_2
\end{array}\]
The relevant part of the fixed locus of $\tau$ on the ambient space breaks into two pieces $\Fix(\tau)=\Fix_1\cup\Fix_2$ and here we have to be a bit careful because of the weighted $\C^*$-action:
\begin{gather*}
\Fix_1=(x_1-x_2=y_1-y_2=z_1-z_2=0),\\
\Fix_2=(x_0=x_1+x_2=y_1-y_2=z_1+z_2=z_0=0).
\end{gather*}
The intersection $Y'\cap\Fix_1$ gives five isolated fixed points of $\tau_0$ contained in the curve $(-2,1,1)\in|K_{Y'}|$. Indeed, according to Lemma \ref{lem!Z3-curve-section}, the intersection of $F_{(-2,1,1)}$ with $y_1=y_2$ on $\Ytilde'$ is six points, but one of these is the intersection of $F$ with the $(-1)$-curve preimage of a basepoint of $|K_{Y'}|$, and this basepoint is actually contained in $Y'\cap\Fix_2$. The following proposition treats $Y'\cap\Fix_2$.
%In here we could compute explicitly the five fixed points of $\tau$ on $Y'$, if necessary.

%{\tiny \marginnote{G: Please add the computation of the points, or erase the comment.}}

\begin{prop}[cf.~\cite{KL}]\label{prop!Z3-fix-R}

The fixed curve $R$ of the involution $\tau$ on $X'$ is a smooth rational curve. There are smooth curves $R_0$ and $R'_0$ in $Y'$ with $g(R_0)=0$ and $g(R'_0)=1$ which intersect transversally at $4$ points and $R_0+R'_0 \in |K_{Y'}|$. Moreover, $R_0^2=-3$ and $R_0'^2=-2$. The image of $R_0$ in $X'$ is $R$ and that of $R'_0$ is $R'$.

\end{prop}
\begin{proof}
The preimage of $R$ on $Y'$ is the $\sigma$-orbit of the intersection $R_0=\Fix_2\cap Y'$. By the above eigenspace decomposition, we see that $R_0$ is a component of the curve $C_0$ in $|K_{Y'}|$ corresponding to $\al_0=0$. Since $Y'$ is in $\cM^s$, we see by Lemma \ref{lem!Z3-curve-section} that $C_0$ is the intersection of the following quadric and cubic in $\PP^3$:
%\begin{gather*}
%(a_{11} - a_{12} + a_{22})t^4 + (b_0y_0 + b_1y_1)t^2 + y_1y_2 = 0,\\
%(c_3 - c_2)t^6 + ((d_3 - d_2)y_0 + (a_{11}-a_{12}+a_{22}- d_3)y_1 +
%(d_2 - a_{11} + a_{12} - a_{22})y_2)t^4 \\
%+ b_0(y_1^2-y_2^2)t^2 + b_1(y_1 - y_0)y_2t^2  - y_0^2(y_1 - y_2) = 0.
%\end{gather*}
%Since $Y'$ is in $\cM^s$, these reduce to
\[
C_0\colon (2a_{11} - a_{12})t^4 + b_0y_0t^2 + y_1y_2 = ((2a_{11} - a_{12} - d_2)t^4 + y_0^2 + b_0t^2(y_1+y_2))(y_1 - y_2)=0.
\]
The cubic is clearly reducible, so $C_0=R_0+R_0'$ has two components, where $R_0$ is a nonsingular conic, and $R_0'$ is an intersection of two quadrics in $\PP^3$ (an elliptic curve). The slice $\Fix_2\cap Y'$ cuts out $R_0$. Note also that $R_0$ and $R_0'$ intersect in four points, thus $C_0$ is a curve of arithmetic genus 4 as expected.

For any given fibre of $\Ytilde'\to\PP^1$, the locus $(t=0)$ gives the intersection of that fibre with the three $(-1)$-curves in $\Ytilde'$. Thus we see from the above equations, that $R_0$ contains one basepoint and $R_0'$ contains two. This implies that $K_{Y'}\cdot R_0=1$, $K_{Y'}\cdot R_0'=2$ and the above stated self-intersection numbers follow from the adjunction formula.
\end{proof}

\begin{prop}
The surface $Z'=X'/\tau$ is rational with five $A_1$ singularities. If $\psi_{X'} \colon X'' \to X'$ is the blow-up of the five points fixed by $\tau$, then the quotient  $\pi_{Z''} \colon X'' \to Z''$ is such that $Z''$ is the minimal resolution of $Z'$, $K_{Z''}^2=-2$, and the divisor $3K_{Z''} +B$ is nef, where $B:=\pi_{Z''}(\psi^*(R))$. The linear system $|3K_{Z''}+B|$ is a base point free pencil of rational curves. There is a sequence of $11$ blowdowns from $Z''$ to $\P^2$ such that the image of $B$ is an (irreducible) plane curve of degree $10$ with five $(3,3)$-points and one $4$-point (i.e., a Campedelli curve).
\label{campe}
\end{prop}

\begin{proof}
Let $E_1,\ldots,E_5$ be the exceptional curves of $\psi_{X'}$, and let $C_i=\pi_{Z''}(E_i)$. The double cover formula reads $$K_{X''} \equiv \ \pi_{Z''}^*\Big(K_{Z''}+\frac{1}{2} (B + C_1 + \dots + C_5)\Big).$$ With this and the formula for the blow-up, one computes $\pi_{Z''}^*(3K_{Z''}+B)\equiv \psi^*(R')$, because $3K_{X'} \equiv R + R'$. Since $R'^2=0$, we have that $3K_{Z''} +B$ is nef. The surface $Z''$ is the minimal resolution of $Z'$, and one computes $K_{Z''}^2=-2$ with the canonical formula above.

Then we satisfy the requirements of \cite[Prop.~3.9]{CCML}, and we belong to the surfaces analyzed in \cite[\S6]{CCML}. In particular, by \cite[Lemma~6.1]{CCML}, the linear system $|3K_{Z''}+B|$ is a base point free pencil of rational curves. The curves $C_1,\ldots,C_5$ are components in the fibres. Since the canonical class of $X'$ is ample and by the possibilities in \cite[Cor.~6.2]{CCML}, the curves $C_i$ are in distinct fibres. Moreover, each of these fibres is a chain of three rational curves: $A_{i,1}+C_i+A_{i,2}$ where the central curve is $C_i$, and the $A_{i,j}$ are $(-1)$-curves. One checks that $B$ is a $6$-section in this pencil, and intersects each $A_{i,j}$ in three points.

We notice that the image of the linear system $|K_{Y'}|$ in $Z'$ defines a pencil of genus $4$ curves with two base points (at the same point). Let $\Delta$ in $Z''$ be the image of $F_{\omega}+F_{\omega^2}$, where $F_{\omega}$ and $F_{\omega^2}$ are the two genus $4$ curves in $|K_{Y'}|$ fixed by the action of $\sigma$, and permuted by $\tau$. Then $\Delta$ is a smooth genus two curve. Let $\Gamma$ in $Z''$ be the image of the curve in $|K_{Y'}|$ which contains the five fixed points of $\tau$ (see description of fixed loci by $\tau$ above). Then $\Gamma$ is a nodal curve of arithmetic genus $2$. Notice that $\Delta$ intersects each of the $A_{i,j}$ at one point transversally, and $\Delta \cdot C_i=0$. Also, $\Gamma$ intersects each $A_{i,j}$ at one point transversally, and $\Gamma \cdot C_i=1$. Notice that $B^2=-6$, $\Delta^2=2$, and $\Gamma^2=2$.

Let $Z'' \to \mathbb{F}_m$ be the blow-down of one of $A_{i,1}$ or $A_{i,2}$ and then $C_i$ for each $i$, so blow-down ten times, where $\mathbb{F}_m$ is a Hirzebruch surface. If $\Gamma_0$ is the $(-m)$-curve in $\mathbb{F}_m$ and $F$ is a fibre, then one verifies
\[\Delta' \sim 2 \Gamma_0 + (m+3)F,  \quad \Gamma' \sim 3 \Gamma_0 + \frac{1}{2}(9+3m)F, \quad B' \sim 6 \Gamma_0 + (3m+7) F\]
where $\Delta',\Gamma',B'$ are the images of $\Delta,\Gamma,B$. Since none of these curves are equal to $\Gamma_0$, we have nonnegative intersection with $\Gamma_0$. Using that, one easily shows $m=1$. This implies that $\Gamma_0 \cdot B'=4$, $\Gamma_0 \cdot \Delta'=2$, and $\Gamma_0 \cdot \Gamma'=3$. Therefore, in $\P^2$, the image of $B$ is an irreducible degree $10$ curve with five $(3,3)$-points and one $4$-point, the image of $\Delta$ is a degree $4$ curve with a node, and the image of $\Gamma$ is a degree $6$ curve with six nodes and one triple point.
\end{proof}

\begin{figure}[htbp]
\centering
\includegraphics[width=16cm]{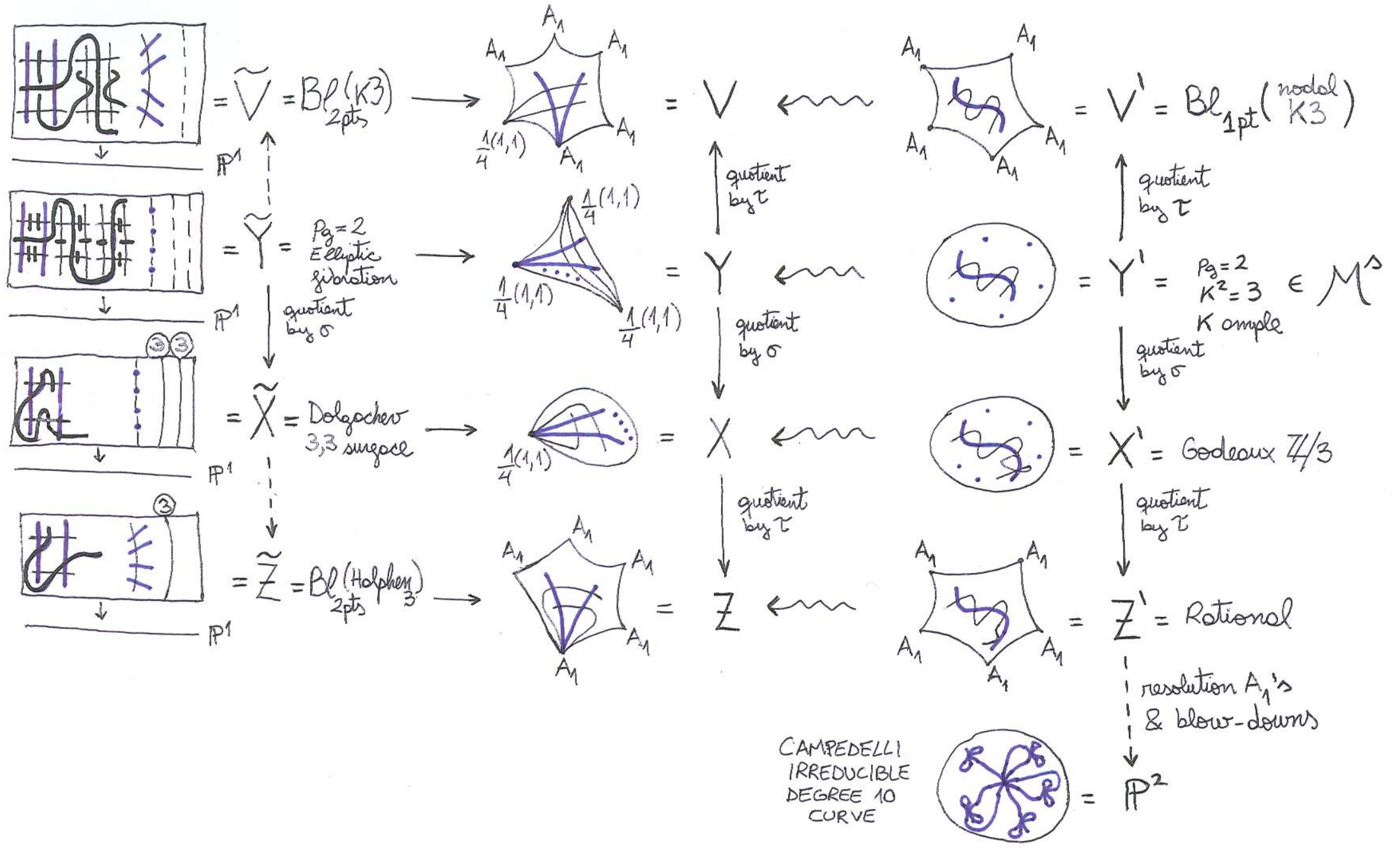}
\caption{Big picture for Sections \ref{s1} and \ref{s2}}
\label{f1}
\end{figure}

\begin{prop}
The minimal model of the quotient $Y'/\tau$ is a K3 surface with five $A_1$ singularities.
\label{V'}
\end{prop}

\begin{proof} We write $\pi_{V'} \colon Y' \to V'$ for the quotient map. By the adjunction formula for double covers, $K_{Y'}=\pi_{V'}^*K_{V'}+R_0$, by Proposition \ref{prop!Z3-fix-R} (see also \cite{KL}), we know that $K_{Y'}\cdot R_0=1$ and $R_0^2=-3$. Thus $\pi_{V'}^*K_{V'}^2=-2$, which implies that $K_{V'}^2=-1$. The residual component $R'_0$ of $C_0$ is also $\tau$-invariant, and $\pi_{V'}|_{R'_0}$ is ramified in the four points where $R'_0$ intersects $R_0$. Since $3=C_0^2=(R_0+R'_0)^2=-3+2\cdot 4 +{R'_0}^2$, we see that ${R'_0}^2=-2$, and hence $\pi_{V'}(R'_0)$ is a $(-1)$-curve. Thus $\pi_{V'}(R'_0)$ is contracted to get $V'_\text{min}$. Since $R_0+R'_0=C \in |K_{Y'}|$, it follows from the projection formula that $K_{V'_\text{min}}$ is numerically trivial, and we also know that $p_g(V'_\text{min})=1$, because we computed the invariants of $\tau$ above. Thus $V'_\text{min}$ is a K3 surface.
\end{proof}

In the right column of Figure \ref{f1}, we show the situation described in this section. In the next section we will describe and use the degenerated situation shown at the left side of Figure \ref{f1}.

%-------------------------------------------------------------------------------------------------
\section{A stable $\Z/3$-Godeaux with $\frac14(1,1)$ singularity} \label{s2}
\subsection{Covering with three $\frac14(1,1)$ and Dolgachev $3,3$ surface}

We now consider the divisor $\cY$ in $\cM$ defined by $b_0=0$, previously considered in \cite{K}. According to Proposition \ref{prop!moduli}, $\cY$ is covered by an 8-dimensional parameter space divided out by a $\C^*$-action. Moreover, the space of admissible parameters is connected and so $\cY$ is irreducible.

\begin{lem} The general member $Y$ in $\cY$ has a free $\sigma$-action and $3\times\frac14(1,1)$ points forming an orbit under $\sigma$, one at each of the coordinate points $P_{y_i}$. After taking the quotient by $\sigma$, we get a stable $\ZZ/3$-Godeaux surface with a single $\frac14(1,1)$ singularity.
\end{lem}

\begin{proof} We work in a neighbourhood of $P=P_{y_1}$. Relations $R_0$, $R_2$ and $S_1$ serve to eliminate variables $y_0,y_2,z_1$ in favour of local coordinates $x_1,x_2,z_2$. Thus $P$ is a hypersurface singularity inside the $\frac12(1,1,1)$ point induced by the $\C^*$-action on the ambient space. The tangent cone of $T_1-b_1R_0$ is a quadric of rank 3 in the local coordinates, and hence $P$ is a $\frac 14(1,1)$ singularity.

The nonsingularity and irreducibility of $Y$ follows by a computer calculation for choices of parameters. The fact that $\sigma$ is fixed point free is a standard computation.
\end{proof}

%{\tiny \marginnote{G: We need to finish the proof... Is $\cY$ irreducible? Need to consider $b_1=0$ and divide out by the group of coordinate transformations to show that $\cY$ is irreducible. }}

%\bigskip
%{\large $\spadesuit:$ Why a general member $Y$ in $\cY$ is irreducible?} Think the only way to prove this is computation that $Y$ is irreducible for a choice of parameters.
%\bigskip

%\bigskip
%{\large $\spadesuit:$ Do you think we have obstructions to deform generic $Y$ in $\cY$? I do not see how to do the computation. It may get easier if we can degenerate further.} I expect no obstructions. I wonder if we can compute the relevant cohomology group using the explicit equations...{\tiny\marginnote{$\leftarrow$ ERASE}}
%\bigskip

%CHECK: This family corresponds to a divisor in the boundary of the moduli space of such Godeaux surfaces.

%---------------------------------------------------------------------------------------------------
%\subsection{Resolution and elliptic surfaces}

Under the degeneration $Y' \rightsquigarrow Y$, the three basepoints of $|K_{Y'}|$ deform to the $\frac14(1,1)$ singular points $P_0,P_1,P_2$ of $Y$. The resolution $\phi_{Y} \colon \Ytilde \to Y$ is given by a single blow up at each of $P_0,P_1,P_2$ with exceptional $(-4)$-curves $E_0$, $E_1$, $E_2$. Indeed,
\[K_{\Ytilde} \equiv \phi_{Y}^*(K_Y)-\frac12(E_0+E_1+E_2),\]
and $\Ytilde$ has an elliptic fibration over $\PP^1$ given by the vector space of global sections
\[H^0(K_{\Ytilde})=\left<x_0,x_1,x_2\right>/{(\sum x_i=0)},\]
where we abuse notation to write $x_i=\phi_{Y}^*(x_i)$. The image of a fibre under $\phi_{Y}$ is a curve of arithmetic genus $4$ with three nodes at basepoints $P_0$, $P_1$, $P_2$.

By the Kodaira formula for the canonical class, $K_{\Ytilde} \sim (\chi(\cO_{\Ytilde})-2)F+\gamma F$, where $F$ is a fibre and $\gamma$ is a term which is zero if and only if there are no multiple fibres. Since $\chi(\cO_{\Ytilde})=3$ and the fibration is induced by $|K_{\Ytilde}|$, we see that there are no multiple fibres, and the $E_i$ are $2$-sections of the fibration. The automorphism $\sigma$ acts on the base of the fibration, and the fibres corresponding to $(1,\omega,\omega^2)$ and $(1,\omega^2,\omega)$ are $\sigma$-invariant. According to Lemma \ref{lem!Z3-curve-section} above, each fibre has a birational model as the intersection of a quadric and a cubic in $\PP^3$.

We consider now the special case when $Y$ has an involution.
\begin{prop}\label{prop!I4-fibre}
The general surface $Y$ in $\cY^s$ with $\Sthree$-action has three fibres of type $I_4$ forming an orbit under $\sigma$, 24 $I_1$ fibres comprising eight orbits, and the two $\sigma$-invariant fibres $F_{\omega}$ and $F_{\omega^2}$ are nonsingular.
\end{prop}

\begin{proof}
Consider the fibre $F_0$ corresponding to $(0,1,-1)$ in $|K_{\Ytilde}|$. Since $b_0=0$ for $Y$ in $\cY^s$, we use the proof of Proposition \ref{prop!Z3-fix-R} to write $\psi_{Y}(F_0)$ as
\[
Lt^4 + y_1y_2=0,\ (Mt^4 - y_0^2)(y_1 - y_2)=0,
\]
where for brevity we write $L=2a_{11}-a_{12}$ and $M=2a_{11}-a_{12}-d_2$.
A short calculation shows that by generality assumption, $F_0$ breaks into four components given by
\[
\renewcommand{\arraystretch}{1.2}
\begin{array}{ll}
F_0^1\colon \lambda t^2+iy_1=y_1-y_2=0, & F_0^2\colon \lambda t^2-iy_1=y_1-y_2=0,\\
F_0^3\colon \mu t^2+y_0=Ly_0^2+My_1y_2=0, & F_0^4\colon \mu t^2-y_0=Ly_0^2+My_1y_2=0
\end{array}\]
where $\lambda^2=L$, $\mu^2=M$. From the equations, we see that $P_0\in F_0^1,F_0^2$, whereas $P_1,P_2\in F_0^3,F_0^4$.

Thus $F_0$ is an $I_4$ fibre $F_0=F_0^1+F_0^3+F_0^2+F_0^4$ where $F_0^1$ and $F_0^2$ are opposite sides, and the $2$-sections intersect $F_0$ as follows: $E_0$ intersects $F_0^1$ and $F_0^2$ transversally, while $E_1$ and $E_2$ intersect $F_0^3$ and $F_0^4$ transversally. See Figure \ref{f1}.

Standard Euler number considerations predict the general $\Ytilde$ has a further 24 nodes in fibres. We compute directly that for general $Y$, these correspond to 24 $I_1$ fibres. Indeed, let $S$ be the disjoint union $S=\bigsqcup\{C : C\in|K_Y|^0\}$, where $|K_Y|^0$ means we ignore the three $I_4$ fibres where one of $\al_i=0$. Then $S\subset\PP^3\times(\PP^1-\{3\text{ points}\})$ is the relative sextic complete intersection of quadric and cubic described in Lemma \ref{lem!Z3-curve-section}(i). Consider the affine piece $S_t$ where $t$ is nonzero, to ignore the three lines of nodes in $S$. The Jacobian subscheme $J$ of $S$ is codimension 2 in the ambient space, defined by $2\times2$ minors of the Jacobian matrix which has entries of degree $\left(\begin{smallmatrix}2 & 2 & 2 \\ 1 & 1 & 1\end{smallmatrix}\right)$. Thus $J$ has relative degree 7. Intersecting with $S_t$ decreases this degree by 3 because of the three nodes on every curve, and so the expected degree of $S_t\cap J$ is $6\times(7-3)=24$. We check using a computer that this intersection is indeed reduced, and comprises 24 points on distinct fibres.

A direct computation using equations \eqref{eq!multiple-fibre} shows that $F_\omega$ and $F_{\omega^2}$ are nonsingular for general $Y$.
\end{proof}

%{\tiny \marginnote{G: You could refer to the big1 picture. Please explain more explicitly the intersections of $E_i$ with the $I_4$.}}

\begin{cor}
The quotient surface $\Xtilde=\Ytilde/\sigma$ is a Dolgachev $(3,3)$-surface, that is an elliptic surface with two fibres of multiplicity $3$. It also has one $I_4$ fibre, eight $I_1$ fibres and a $6$-section $E$.
\end{cor}

%{\tiny \marginnote{G: This should be clarified by the proof of the Proposition above.}}

%\bigskip
%{\large $\spadesuit:$ The two multiple fibers are smooth, right?} Yes, for general $Y$. Need a proper proof of this, related to previous remark about $I_1$ fibres. Please add this in the above Proposition.
%\bigskip

%Standard Euler number considerations show that the generic $\Ytilde$ has 36 singular fibres, but for our special choice, we get an orbit of three $I_4$ fibres. Indeed, we get singular fibres of $\Ytilde$ corresponding to elements of the pencil $|K_Y|$ that are singular at more than just the three base points of the pencil. Now, $|K_Y|$ is spanned by the $x_i$, subject to the condition $x_0+x_1+x_2=0$.
%\begin{align*}
%-x_0z_0 + y_1y_2 - r_0x_1x_2 &= 0, \\
%-x_1z_1 + y_0y_2 - r_1x_0x_2,\\
%-x_2z_2 + y_0y_1 - r_2x_0x_1,\\
%-y_0z_0 + r_1x_2y_1 + r_2x_1y_2 + Sx_1x_2 &= 0,\\
%-y_1z_1 + r_0x_2y_0 + r_2x_0y_2 + Sx_0x_2,\\
%-y_2z_2 + r_0x_1y_0 + r_1x_0y_1 + Sx_0x_1,\\
%-z_1z_2 + r_0y_0^2 + Sx_0y_0 + r_1r_2x_0^2 &= 0\\
%-z_0z_2 + r_1y_1^2 + Sx_1y_1 + r_0r_2x_1^2,\\
%-z_0z_1 + r_2y_2^2 + Sx_2y_2 + r_0r_1x_2^2
%\end{align*}

\begin{thm}\label{topoteo}
The topological fundamental group of a $\Z/3$-Godeaux surface is $\Z/3$.
\end{thm}

\begin{proof}
The action of $\sigma$ in a $\Q$-Gorenstein degeneration $Y' \rightsquigarrow Y$ over a disc is such that the quotient is a $\Q$-Gorenstein degeneration $X' \rightsquigarrow X$ over a disc. From $X$ to $X'$, we are $\Q$-Gorenstein smoothing the $\frac{1}{4}(1,1)$ singularity of $X$. On the other hand, the resolution of $X$ is a Dolgachev $(3,3)$-surface $\Xtilde$, and so $X$ has topological fundamental group $\pi_1(X)$ isomorphic to $\Z/3$. %Let us use the notation $\pi_1$ for topological fundamental group.

We now follow the strategy in \cite{LP07}. Using Seifert--van-Kampen theorem, we have that $\pi_1(X) \simeq \Big(\pi_1(X \setminus{E}) * \pi_1(U)\Big) /\langle \gamma \rangle_n$ where $U$ is a small neighborhood of the singularity in $X$, and so $\pi_1(U)$ is trivial, $\gamma$ is a loop around $E$ in $X \setminus E$, and $\langle \gamma \rangle_n$ is the smallest normal subgroup in $\pi_1(X \setminus{E})$ containing $\gamma$. But by Proposition \ref{prop!I4-fibre}, there is a $\P^1$ from the $I_4$ fibre in $\Xtilde$, which intersects the $6$-section $E$ transversally at one point, and so $\gamma$ is trivial in $\pi_1(X \setminus{E})$. Therefore $\pi_1(X \setminus{E}) \simeq \Z/3$. Now, since $X' \rightsquigarrow X$ is $\Q$-Gorenstein smoothing, the surface $X'$ is homeomorphic to $X \setminus{E}$ union the Milnor fibre $M_f$ of $\frac{1}{4}(1,1)$ corresponding to the smoothing, glued along the corresponding link $L$. So we apply Seifert--van-Kampen theorem again to $X'$ using that decomposition. Since the generator $\gamma$ of $\pi_1(L)$ is trivial in $\pi_1(X \setminus{E})$ and the inclusion induces $\pi_1(L)=\Z/16 \twoheadrightarrow \pi_1(M_f)=\Z/4$, we obtain that $\pi_1(X') \simeq \Z/3$. Similarly, one can use the same strategy with the $\Q$-Gorenstein degeneration $Y' \rightsquigarrow Y$ and the elliptic fibration $\Ytilde$, to prove that $Y'$ is simply-connected, and so $\pi_1(X') \simeq \Z/3$.
\end{proof}

%$$\begin{CD}
%\Ytilde @>\phi_Y >>  Y\\
%@VV{\pi_{\Xtilde}}V  @VV{\pi_{X}}V\\
%\Xtilde @>\phi_{X}>>  X\\
%\end{CD}$$
\subsection{Stable Campedelli double plane, K3 quotient and its flipped family} \label{flipped}

As before, since the involutions $\tau_i$ are conjugate under the action of $\sigma$, we use $\tau=\tau_0$ for our computations. The following explains how the involution descends to $\Xtilde$.

\begin{prop}
The action of $\tau$ on $\Xtilde$ has two invariant fibres, and swaps the two fibres of multiplicity three. One of the invariant fibres is the $I_4$ fibre, of which two components are fixed pointwise. The other invariant fibre is nonsingular and contains four isolated fixed points.
\end{prop}

\begin{proof}
From Proposition \ref{prop!I4-fibre}, the intersections of $\Fix_i$ with $Y$ are:
\begin{align*}
Y\cap\Fix_1 &= P_0\cup\text{four points on the fibre }F_{(-2,1,1)},\\
Y\cap\Fix_2 &= F_0^1 + F_0^2=: R_0.
\end{align*}
We can check that $F_{(-2,1,1)}$ is nonsingular directly using Lemma \ref{lem!Z3-curve-section}. Moreover, the two $\sigma$-invariant fibres are exchanged by $\tau$. The involution must act nontrivially on the $6$-section $E$, fixing two points which are the intersection points of $E$ with $\Fix_2$ from above.

The 2-sections $E_1$ and $E_2$ on $\Ytilde$ form an orbit under $\tau$, while $E_0$ is invariant with two fixed points at the intersection with the $I_4$ fibre.
\end{proof}

\begin{cor}
The quotient $\Xtilde/\tau$ has four $A_1$ singularities on one fibre. Its resolution is the blow-up at two points of a relatively minimal rational elliptic surface with one fibre of multiplicity three. The singular fibres are $I_2$, $I_0^*$, and four $I_1$. Thus a contraction of $\Xtilde/\tau$ to $\PP^2$ gives a Halphen pencil of index 3.
\end{cor}

We now describe the corresponding stable Campedelli double plane which is an analogue of Proposition \ref{campe}. We introduce some notation. Let $\widehat{X} \to \Xtilde$ be the blow-up at the four points fixed by $\tau$, and let $\Ztilde$ be the minimal resolution of $\Xtilde/\tau$. Let $\pi_{\Ztilde} \colon \widehat{X} \to \Ztilde$ be the quotient by $\tau$. Let $B_1,B_2$ be the image by $\pi_{\Ztilde}$ of the fixed curve by $\tau$ in $\Xtilde$, and let $C_1,C_2,C_3,C_4$ be the image of the exceptional curves of $\widehat{X} \to \Xtilde$. Thus, $B_1+B_2+C_1+\ldots+C_4$ is the branch curve of $\pi_{\Ztilde}$. Let $A_1+A_2+B_1+B_2$ be the image of the $I_4$ fibre by $\pi_{\Ztilde}$, and so $A_1,A_2$ and $B_1,B_2$ are opposite sides of the new $I_4$, and $A_i^2=-1$, $B_i^2=-4$. Let $E_0$ be the image by $\pi_{\Ztilde}$ of the $(-4)$-curve from $E$ in $\Xtilde$. Let $\Gamma$ be the image by $\pi_{\Ztilde}$ of proper transform of the fibre in $\Xtilde$ which contains the four fixed points by $\tau$. Finally, let $\Delta$ be the image by $\pi_{\Ztilde}$ of the two multiple fibres in $\Xtilde$.

\begin{prop}
The linear system defined by $A_1+E_0+A_2 \sim 3K_{\Ztilde} + E_0+B_1+B_2$ defines a genus $0$ fibration $f \colon \Ztilde \to \P^1$, which pulls back to $\Xtilde$ as a genus $2$ fibration. Each $C_i$ belongs to one fibre of $f$ which is formed by $(-1)$-curves $A_{i,1}$, $A_{i,2}$ and $C_i$. The blow-down of $A_1,A_{1,1},\ldots,A_{4,1}$ and $E_0,C_1,\ldots,C_4$ is the blow-up at one point of $\P^2$. After blowing-down to $\P^2$ the curves $B_1$, $B_2$ become two quintics such that $B_1+B_2$ has four $(3,3)$-points, one $4$-point, and one singular point of the form $\{(y^2-x^2)(y^2-x^4)=0\}$ locally at $(0,0) \in \C_{x,y}^2$. We also realize $\Delta$ as a quartic with two nodes, and $\Gamma$ as a sextic with two triple points and four double points.
\label{stablecampe}
\end{prop}

\begin{proof}
The linear system $|A_1+E_0+A_2|$ defines a rational fibration in $\Ztilde$. The curve $B_1+B_2$ is a $6$-section, and so the pull back is a fibration of genus $2$ curves. Notice that the curves $C_i$ are in fibres of $|A_1+E_0+A_2|$. The options for the irreducible components of the fibre containing $C_i$ are as in \cite[Cor.~6.2]{CCML}. In $\Xtilde$ a $(-2)$-curve has to intersect the $6$-section $E$ since $K_{X}$ is ample. Therefore, as in Proposition \ref{campe}, the only possible irreducible components for a fibre containing $C_i$ are two $(-1)$-curves $A_{i,1}$, $A_{i,2}$, so that they form a chain with $C_i$ as central curve. Since after blowing down $A_1$, $A_2$ we obtain a relatively minimal Halphen fibration of index $3$, one can check that $\Gamma$ intersects each $A_{i,j}$ transversally at one point. Notice that $\Gamma \cdot C_i=1$ as well, and $\Gamma$ is a $3$-section of $|A_1+E_0+A_2|$. Also, $\Delta \cdot A_{i,j}=1$ for all $i,j$, and $\Delta$ is a $2$-section of $|A_1+E_0+A_2|$. We also have $B_j \cdot A_i=1$, $B_i \cdot E_0=1$, $B_1 \cdot B_2=0$, $\Delta^2=0$, and $\Gamma^2=-2$. Since $K_{\Ztilde}^2=-2$, after blowing down  $A_1,A_{1,1},\ldots,A_{4,1}$ and $E_0,C_1,\ldots,C_4$ we arrive to a Hirzebruch surface $\F_m$. Since we know the self-intersections of the images of $\Delta$, $\Gamma$, and $B_1+B_2$, we obtain as in Proposition \ref{campe} that $m=1$.

Now the point is that $B_1+B_2$ has degree $10$ in $\P^2$, but there are possibilities for the degrees of $B_1$ and $B_2$. Since $A_1$, $E_0$ gives nodes to $B_1$ and $B_2$, the degree cannot be smaller than $3$. The possibilities of $B_k \cdot A_{i,1}$ are $0,1,2,3$. We check that the only possibility that works is $B_k$ intersects two $A_{i,1}$ at two points each (possible infinitely near) and the other two at one point each. Therefore, each $B_k$ becomes a quintic in $\P^2$. The image in $\P^2$ of $\Delta$ is a quartic with two nodes, and the image of $\Gamma$ is a sextic with two triple points and four double points.
\end{proof}

\begin{prop}
The minimal model of the quotient $\Ytilde/\tau$ is a K3 surface, which has an elliptic fibration with singular fibres $I_2$, $I_4$ and $I_0^*$, a $(-2)$-curve which is a $2$-section, and a $(-4)$-curve which is a $4$-section.
\label{V}
\end{prop}

\begin{proof}
Let us consider the $I_4$ fibre $F_0=\sum_{i=1}^4 F_0^i$ that is $\tau$-invariant ($=\tau_0$). Under the quotient by $\tau$, the two curves $F_0^1$ and $F_0^2$ that are pointwise fixed go to $(-4)$-curves, while $F_0^3$ and $F_0^4$ of go to $(-1)$-curves. After contracting the $(-1)$-curves, we get an $I_2$ fibre. The other two $I_4$ fibres are identified to give a single $I_4$ fibre. The nonsingular invariant fibre maps to an $I_0^*$ after resolving the four $A_1$ singularities. The two $(-4)$-curves $E_1$ and $E_2$ which form a $\tau$-orbit are identified to give a $(-4)$-curve which is a $4$-section. After blowing down $F_0^3$ and $F_0^4$, it passes through the two nodes of the $I_2$-fibre. The image of the single $(-4)$-curve $E_0$ is a $(-2)$-curve which is a $2$-section. The multiplicity $2$ of the $(-1)$-components disappears after contracting and stays in the central curve of the $I_0^*$, so there are no multiple fibres. Notice that $p_g=1$, and so the minimal model of $\Ytilde/\tau_0$ is a K3 surface.
\end{proof}

Let $\pi_{V} \colon Y \to V$ be the quotient by $\tau$. Notice that $V$ has five $A_1$ singularities (four from the four fixed points, and one from two of the $\frac{1}{4}(1,1)$ singularities), and one $\frac{1}{4}(1,1)$ singularity. The minimal resolution $\Vtilde$ of $V$ is the minimal resolution of $\Ytilde/\tau$. Using the notation in the proof of Proposition \ref{V}, let $D_1$ and $D_2$ be the images of $F_0^3$ and $F_0^4$ in $V$. Notice that in $\Vtilde$, the proper transforms of $D_1$ and $D_2$ are $(-1)$-curves intersecting transversally at one point the $(-4)$-curve which is the exceptional divisor of the $\frac{1}{4}(1,1)$ singularity. See Figure \ref{f1} for this and what follows.

We have $V' \rightsquigarrow V$, where $V':=Y'/\tau$ is the blow-up at one point of a nodal K3 surface, as in Proposition \ref{V'}. Let us consider that degeneration over a disc $\D$. This is, we consider the $\Q$-Gorenstein deformation $(V \subset \cV) \to (0 \in \D)$ where $V'$ is a fibre for some $t \in \D \setminus \{0\}$. Let $V \to U$ be the contraction of the curve $D_1$. Then, there is a blow-down deformation $(U \subset \cU) \to (0\in \D)$ which commutes with $\cV \to \D$ over $\D$. This and what follows is explained in \cite{HTU}, see also \cite[\S2]{Urz}. In this way, we have an extremal neighborhood $(D_1 \subset \cV) \to (Q \in \cU)$ of type k1A. Notice that $(Q \in U)$ is a $\frac{1}{3}(1,1)$ singularity. This extremal neighborhood is of flipping type, and it is the simplest case among all k1A (see \cite[Prop.2.15]{Urz}). After we perform the flip, we obtain a Gorenstein deformation $(V_1 \subset \cV_1) \to (0 \in \D)$ where $V_1$ is the minimal resolution of $U$. Hence the flipping curve is the $(-3)$-exceptional curve. Since this was a flip, the fibre $V'$ in $\cV$ appears also in $\cV_1$.

\begin{figure}[htbp]
\centering
\includegraphics[width=16cm]{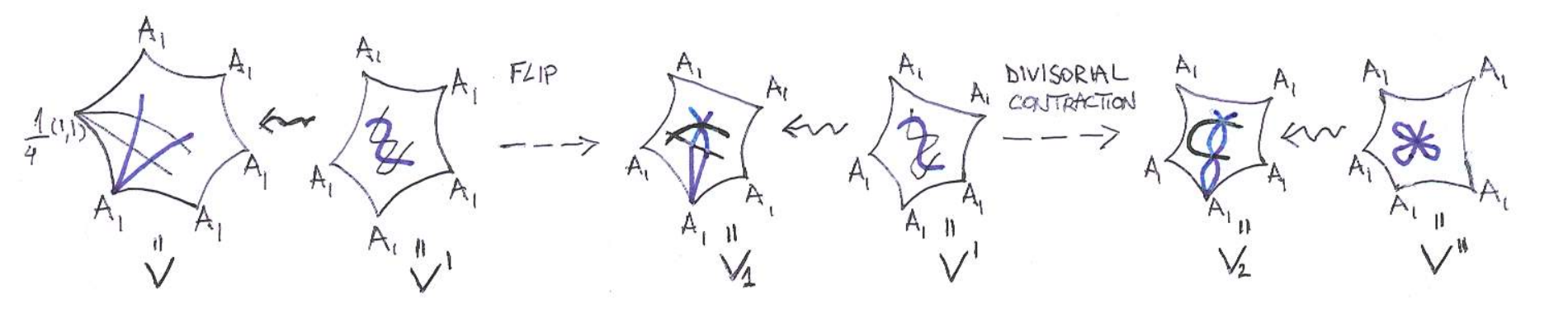}
\caption{Birational transformations on K3 family}
\label{f2}
\end{figure}

After that, the proper transform of $D_2$ together with the $(-1)$-curve in $V'$ generate a contractible divisor in $\cV_1$. After this divisorial contraction, we obtain a Gorenstein deformation $(V_2 \subset \cV_2) \to (0 \in \D)$ where $V_2$ is the blow-down of $D_2 \subset V_1$, and the general fibre is a nodal K3 surface $V'_1$. Notice that the branch curve of $Y' \to V'$ in $V'_1$ has a $4$-point, and degenerates nontrivially to a nonreduced curve formed by three irreducible components (one with multiplicity $2$), two passing through one $A_1$, and the three together having two 3-points in the smooth locus of $V_2$, as shown in Figure \ref{f2}. We have proved:

\begin{prop}
The $\Q$-Gorenstein deformation $V' \rightsquigarrow V$ over a disk $\D$ is birationally equivalent over $\D$ to a Gorenstein deformation $V'' \rightsquigarrow V_2$, where $V''$ is the blow-down of a $(-1)$-curve in $V'$, and $V_2$ is the K3 surface with five nodes obtained after the flip and divisorial contraction explained above.
\label{flip}
\end{prop}

%--------------------------------------------------------------------------------------------------------------
\subsection{Stable rational $\Z/3$-Godeaux surfaces} \label{s23}

A known method for producing $\frac{1}{n^2}(1,n-1)$ singularities on an elliptic fibration is to merge a nodal fibre with a multiple fibre of multiplicity $n$; see e.g.~\cite{K92}, \cite[\S4]{Urz}. We apply this method with $n=3$ to the Dolgachev $(3,3)$ surfaces $\Xtilde$. To achieve that, we first merge a nodal fibre in $\Ytilde$, the universal cover of $\Xtilde$, with a $\sigma$-fixed fibre so that $\Ytilde$ acquires an $A_2$ singularity. Then the quotient by $\sigma$ gives what we want. We can do this independently with each of the two fixed fibres. In fact, this approach extends to $Y'$ in $\cM$ with a single $A_2$-singularity that is fixed by $\sigma$. For such $Y'$, one of the $\sigma$-invariant genus 4 fibres of the canonical pencil $\Ytilde'\to\PP^1$ acquires a node.

%{\tiny \marginnote{G: In this section there is a lot to do. Please fill out details.}}

\begin{prop}\label{prop!B-families}
There are two divisors $\cB_{\omega}$ and $\cB_{\omega^2}$ each defined by hypersurfaces in $\cM$ corresponding to surfaces $Y$ with an $A_2$ singularity that is fixed by $\sigma$. The quotient $X$ of such a surface is a Godeaux surface with a $\frac19(1,2)$ singularity.
\end{prop}

\begin{proof}
We consider the $\sigma$-invariant fibre $F_\omega$ on $Y$, corresponding to $(1,\omega,\omega^2)$ in $H^0(K_Y)$. The computation for $F_{\omega^2}$ is similar. Suppose $F_\omega$ has a node $Q$. Since $Q$ is a distinguished point of $F_\omega$, it is fixed by $\sigma$. Thus by Lemma \ref{lem!Z3-curve-section} and the discussion following, $Q=(1,\la,\omega^2\la,\omega\la)$ in $\PP^3$ for some $\la$. We substitute $Q$ into the equations for $F_\omega$ and rearrange a little to get the following two conditions:
\begin{equation}\label{eqn!B-family}
\begin{split}
\la^2 = (b_0 + \omega^2 b_1)\la + \omega^2 a_{11} + a_{12} + \omega a_{22} \\
2\la^3 = -3((\omega^2 d_2 + \omega d_3)\la + \omega c_2 + \omega^2 c_3).
\end{split}
\end{equation}
Eliminating $\la$ gives a single hypersurface in $\cM$ defining $\cB_\omega$. The other divisor $\cB_{\omega^2}$ is obtained by interchanging $\omega$ and $\omega^2$.

It remains to check that the general surface in $\cB_\omega$ has a single $A_2$ singularity and no others. Consider a neighbourhood $U$ of $Q\in F_\omega$ in $Y$. We can write $U$ as the complete intersection in $\A^3\times\Delta_s$ given by the equations of Lemma \ref{lem!Z3-curve-section}(i) with $\al_1=\omega+s$ and $\al_2=\omega^2-s$, where $s$ is the coordinate on the disc $\Delta$ and we assume $t\ne0$. A careful (but tedious) analysis of the two equations confirms that $Q$ is an $A_2$ singularity for the general surface in $\cB_\omega$.
%{\tiny \marginnote{S: A little left to do here, nothing serious}}
\end{proof}

\begin{rem} The hypersurface defining $\cB_\omega$ is too large to reproduce here, so we consider $\cB^s_\omega$, the restriction of $\cB_\omega$ to $\cM^s$. In fact, we have $\cB^s_\omega=\cB^s_{\omega^2}=(\cB_\omega\cap\cB_{\omega^2})^s$, because the involution maps the first $A_2$ singularity to a second $A_2$ singularity. The hypersurface in $\cM^s$ defining $(\cB_\omega\cap\cB_{\omega^2})^s$ is
%too large to reproduce here, but the intersection $\cB_\omega\cap\cM^s$ is given by
\[\frac 23 (a_{11}-a_{12})^3 + 2(a_{11}-a_{12})^2d_2 + (a_{11}-a_{12})(3b_0c_2-b_0^2d_2+\frac32d_2^2) + \frac12 c_2(3c_2+3b_0d_2-2b_0^3).
\]
\end{rem}
%{\tiny \marginnote{G: Please elaborate on the details. The independency of smoothing should be clarified in some way.}}

\begin{lem}\label{lem!Z3-transversal} All intersections between $\cY$, $\cB_\omega$, and $\cB_{\omega^2}$ are transversal. Thus, for a general $Y$ in any of these intersections, each singularity of $Y$ can be $\Q$-Gorenstein smoothed independently, and the same applies to the Godeaux quotient $X$, see Figure \ref{f3} below.
\end{lem}
\begin{proof}
Let $Y$ be a general surface in $\cY\cap\cB_\omega$. For simplicity, we work analytically on $\cB_\omega$, but the Lemma also holds in the algebraic setting. From equation \eqref{eqn!B-family} above, we can solve for $\la$ using the quadratic, and substitute into the cubic to get an analytic hypersurface defining $\cB_\omega$.

To smooth the $\frac14(1,1)$ singularity, let $Y_\varepsilon$ be the surface with parameters $b_0=\varepsilon$, $b_1=b_1-\varepsilon\omega$ and all others the same as those for $Y$. Thus $Y_\varepsilon$ remains in $\cB_\omega$, but is no longer in $\cY$. To smooth the $\frac19(1,2)$ singularity, simply vary one of $a_{ij}$ or $b_1$ while fixing $b_0=0$ and the other parameters, so that any root of the quadratic is no longer a root of the cubic in \eqref{eqn!B-family}.
\end{proof}

\begin{figure}[htbp]
\centering
\includegraphics[width=14cm]{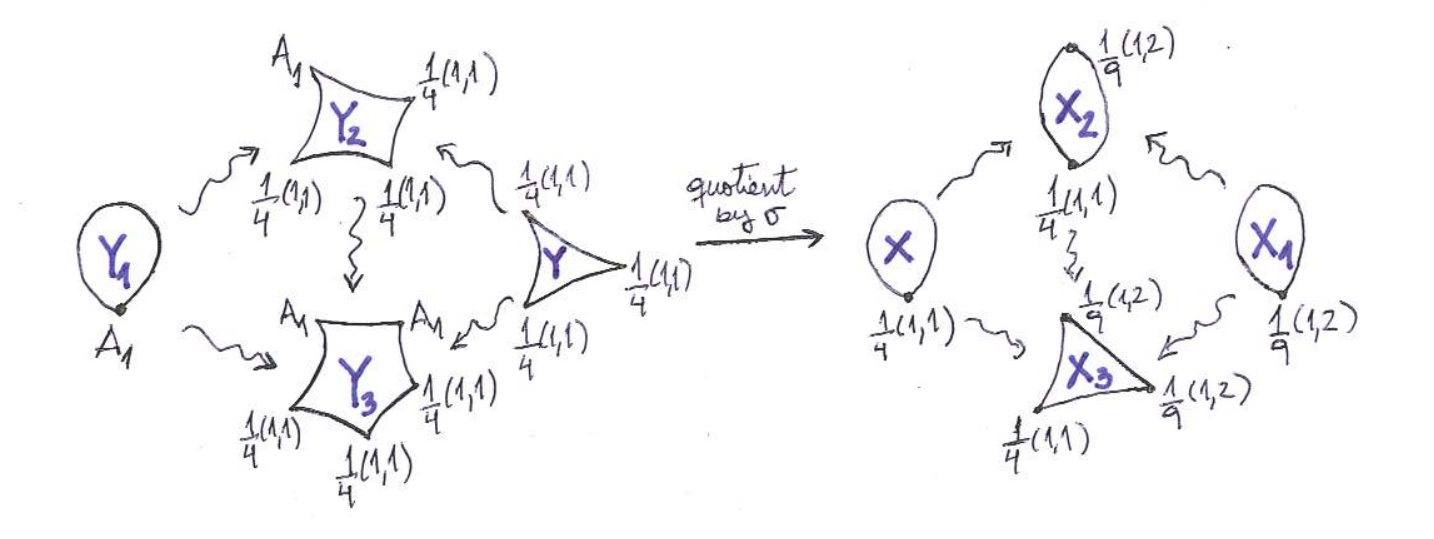}
\caption{Stable $\Z/3$-degenerations of $Y$ and $X$}
\label{f3}
\end{figure}

Let $X_i$ for $i=1,2,3$ be stable Godeaux surfaces such that $X_1$ has one $\frac{1}{9}(1,2)$ singularity, $X_2$ has one $\frac{1}{4}(1,1)$ and one $\frac{1}{9}(1,2)$, and $X_3$ has one $\frac{1}{4}(1,1)$ and two $\frac{1}{9}(1,2)$ singularities. Let $Y_i$ be the corresponding cyclic cover of $X_i$. All of them relate under $\Q$-Gorenstein degenerations as shown in Figure \ref{f3}.

\begin{prop}\label{prop!rational-X23}
The surfaces $X_2$ and $X_3$ are rational.
\end{prop}

\begin{proof}
Let $i=2,3$. By construction, if $\widetilde Y_i \to Y_i$ is the minimal resolution of the three $\frac{1}{4}(1,1)$ singularities, then there is elliptic fibration $\widetilde Y_i \to \P^1$ such that the $A_2$ singularity(ies) is(are) in the node of an(two) $I_1$ fibre(s) $F_\omega$ and/or $F_{\omega^2}$ that are $\sigma$-invariant. The quotient $\widetilde X_i:= \widetilde Y_i/\sigma$ is the minimal resolution of $X_i$ at the $\frac{1}{4}(1,1)$ singularity. Therefore, there is an elliptic fibration $X_i \to \P^1$ with an(two) $I_1$ fibre(s) such that the $\frac{1}{9}(1,2)$ of the surface is(are) at the node(s), and that (those) fibre(s) is(are) multiple with multiplicity $3$. Take one $\frac{1}{9}(1,2)$ singularity and minimally resolve. Then the multiple fibre becomes a $I_3$ fibre in the induced elliptic fibration. Notice that $K_{\widetilde X_i}^2=0$, and so the self-intersection of the canonical class in the resolution is $-2$. Therefore the proper transform of the $I_1$ is a $(-1)$-curve.

In the case $i=2$, we obtain an elliptic fibration with $p_g=q=0$ and one multiple fibre of multiplicity $3$. This corresponds to a Halphen pencil of index $3$ in $\P^2$. In the case $i=3$, we obtain an elliptic fibration with $p_g=q=0$ and no multiple fibre. This corresponds to a pencil of cubics in $\P^2$. So in both cases they are rational surfaces.

%\bigskip
%{\large $\spadesuit:$ To finish the proof, I need a more explicit model for $X_3$. I may get it if you can give me some other curve in the surface. For example, it would be extremely useful to have sections for the elliptic fibration $\widetilde X_3$, maybe it is not difficult to get them...}
%\bigskip

\end{proof}

\begin{prop}
The surface $X_1$ is rational.
\end{prop}

\begin{proof}
Assume $X_1$ is not rational. Let $\widetilde X_1$ be its minimal resolution, and let $S$ be the minimal model of $\widetilde X_1$. We know that $S$ cannot be of general type since $K_{X_1}^2=1$ and $K_{X_1}$ is ample; see e.g.~\cite[Prop.~3.6]{Urz}. Therefore $K_S^2=0$. Also, we have that $K_{\widetilde X_1}^2=1-2=-1$, and so $S$ is the blow down of one $(-1)$-curve $E \subset \widetilde X_1$. Let $\Gamma_1$ be the $(-2)$-curve in the exceptional divisor of $\widetilde X_1 \to X_1$, and let $\Gamma_2$ be the $(-5)$-curve. Since $K_S$ is nef, we have that $E$ cannot touch $\Gamma_1$. Let $m=E \cdot \Gamma_2$. Since $K_{X_1}$ is ample, we have $m \geq 2$. Since $K_S$ is nef, we have $m \leq 3$. If $m=3$, then $S$ cannot have Kodaira dimension $1$ since $K_S \cdot \Gamma'_5=0$ and $\Gamma'_2$ is not a fibre of an elliptic fibration (image of $\Gamma_2$). Therefore $S$ must be an Enriques surface since $p_g(X_1)=q(X_1)=0$. Then $\pi_1(X_1)=\Z/2$. On the other hand, we have a $\Q$-Gorenstein smoothing $X' \rightsquigarrow X_1$ over a disk $\D$, where $X'$ is a nonsingular $\Z/3$-Godeaux surface, and so we have a surjective morphism $\pi_1(X')=\Z/3 \to \pi_1(X_1) = \Z/2$, which is a contradiction.

Thus $m=2$. Let us again consider the $\Q$-Gorenstein smoothing $X' \rightsquigarrow X_1$ over $\D$ above. We notice that a loop $\alpha$ around $\Gamma_2$ will satisfy $\alpha^2=1$ in $\pi_1(\widetilde X_1 \setminus \{\Gamma_1 \cup \Gamma_2\})$, because $\Gamma_2 \cdot E=2$, $\Gamma_1 \cdot E=0$, and $E=\P^1$. But $\alpha^9=1$ as well, and so $\alpha$ is trivial in $\pi_1(\widetilde X_1 \setminus \{\Gamma_1 \cup \Gamma_2\})$. By the same Seifert--Van-Kampen argument in Theorem \ref{topoteo}, this would imply $\pi_1(X_1)=\pi_1(X')=\Z/3$. Then $S$ is a Dolgachev surface with two multiple fibres of multiplicities $3a$ and $3b$ where $\gcd(a,b)=1$. Using the canonical class formula and $K_S \cdot \Gamma'_2=1$, we get $ \frac{1}{3a} + \frac{1}{3b} + \frac{1}{3at} =1$ where $3at$ is the degree of the multi-section $\Gamma'_2$ in the elliptic fibration of $S$. Then $a=b=t=1$, and $\Gamma'_2$ is a $3$-section.

We now consider the $\Q$-Gorenstein degeneration $X_1 \rightsquigarrow X_2$ over $\D$ above. Since around the singularity $1/9(1,2)$ this deformation is trivial, we resolve simultaneously to obtain a $\Q$-Gorenstein smoothing $X'_1 \rightsquigarrow X'_2$ over $\D$ (of the singularity $1/4(1,1)$). Since $K_{X'_1}^2=-1$, the canonical class of the $3$-fold is not nef. As in \cite[\S 2]{Urz2}, we either have a divisorial contraction or a flip for the family (the other ``ending options" are not possible). In case of a flip, the central fiber would become smooth, and then there would be a contradiction since the Kodaira dimension must be constant in a smooth deformation (and $X_2$ is rational by the previous proposition). So we have a divisorial contraction, and it must correspond to a $(-1)$-curve in $X'_2$ which does not touch the singularity $1/4(1,1)$. But then in the general fiber this $(-1)$-curve propagates as the curve $E$ above. So we contract this divisor in the family, to obtain another $\Q$-Gorenstein smoothing $X''_1 \rightsquigarrow X''_2$ over $\D$, where $X''_1$ is a Dolgachev $3,3$ surface. But now the canonical class of the $3$-fold must be nef (again flips are not allowed by the above argument) and so, by \cite[Theorem 5.1]{K92}, the canonical class of this new $3$-fold has index $3$ (or $1$), but this is locally a $\Q$-Gorenstein smoothing of $1/4(1,1)$, which has index $2$, a contradiction.

\end{proof}

%\bigskip
%{\large $\spadesuit:$ I would like to give a $\P^2$ model for the one with 3 singularities. For that I need more curves, maybe just one more. I already have possibilities. After that, I would be able to degenerate further.}
%\bigskip

%\bigskip
%{\large $\spadesuit:$ The situation of $\Ytilde \to \Xtilde$ on the fixed fibre is locally $A_{3d-1} \to \frac{1}{3d}(1,3d-1)$. You have found $d=1$. Possibilities in our case are $d=1,2,3,4,5,6,7$. Can you find $A_2,\ldots,A_{20}$? This is not too important, but maybe it is possible. This is like finding a $\Ytilde$ with an $I_{21}$ fibre. Maybe you can do it using discriminant if available.}
%\bigskip

%-------------------------------------------------------------------------------------------------

%--------------------------------------------------------------------------------------------------------------

%--------------------------------------------------------------------------------------------------------------
\section{On $\ZZ/4$-Godeaux surfaces} \label{s3}
\subsection{Setup and involution}\label{sec!Z4-smooth}

Start with $\PP(1,1,1,2,2)$ with coordinates $x_1,x_2,x_3,y_1,y_3$ and a $\ZZ/4$-action
\[\sigma\colon (x_1,x_2,x_3,y_1,y_3)\mapsto(ix_1,-x_2,-ix_3,iy_1,-iy_3).\]
Now let $Y'$ be the intersection of two quartics $q_0$ and $q_2$ in eigenspaces $0$, $2$ of the form
\begin{align*}
q_0 &\colon x_1^4+x_2^4+x_3^4+x_1^2x_3^2+a_1x_1x_2^2x_3+a_2x_1x_2y_1+a_3x_2x_3y_3+a_4y_1y_3,\\
q_2 &\colon x_1^2x_2^2+x_2^2x_3^2+y_1^2+y_3^2+b_1x_1^3x_3+b_2x_1x_3^3+b_3x_1x_2y_3+b_4x_2x_3y_1.
\end{align*}
Reid \cite{R} (see also \cite{Miy}) showed that for general choices of parameters $a_i,b_i$, the surface $Y'$ is the nonsingular canonical model of general type with $p_g=3$, $K^2=4$ and a free $\ZZ/4$-action, whose quotient $X'$ is a Godeaux surface with $\pi_1=\ZZ/4$.  Their coarse moduli space is 8-dimensional, irreducible and unirational. As before, we use $\cM$ for the KSBA moduli space of stable $\ZZ/4$-Godeaux surfaces.

\begin{rem} For future use, we want to allow certain monomials to appear with coefficient zero, thus we do not use the above displayed parameters.
\end{rem}

Following Keum--Lee \cite[\S4.3]{KL}, there is a 4-dimensional family $\cM^s \subset \cM$ with an involution $\tau$ on $X'$ induced by the following action on $Y'$:
\begin{equation}\label{eqn!Z4-tau}
\tau\colon (x_1,x_2,x_3,y_1,y_3)\mapsto(-x_1,x_2,-x_3,y_1,y_3).
\end{equation}
The group generated by $\tau$ and $\sigma$ is $\ZZ/2\times\ZZ/4$. For surfaces in $\cM^s$, the allowed monomials for $q_0$ and $q_2$ are
\begin{align*}
q_0 &\colon x_1^4,\ x_2^4,\ x_3^4,\ x_1^2x_3^2,\ x_1x_2^2x_3,\ y_1y_3,\\
q_2 &\colon x_1^2x_2^2,\ x_2^2x_3^2,\ x_1^3x_3,\ x_1x_3^3,\ y_1^2,\ y_3^2.
\end{align*}

Let $Y'$ in $\cM^s$ be the nonsingular cover of a $\Z/4$-Godeaux $X'=Y'/\sigma$. The following Lemma is an easy version of Proposition \ref{prop!Z3-fix-R}.

\begin{lem}\label{lem!Z4-fix-tau}
The fixed curve $R_0$ of $\tau$ on $Y'$ is the curve of genus $5$ defined by $(x_2=0)$. The five $\sigma$-orbits $Q_1,\dots,Q_5$ on $Y'$ corresponding to the five isolated fixed points $P_1,\dots,P_5$ of $\tau$ on $X'=Y'/\sigma$ are given by $Q_1\colon(x_1=x_3=0)$ and $Q_2,\dots,Q_5\colon(y_1=y_3=0)$. The $\sigma$-orbit $Q_1$ is pointwise fixed by $\tau$, whereas $Q_2,\dots,Q_5$ are only $\tau$-invariant.
\end{lem}

\begin{rem} The Classification Theorem of \cite{CCML} implies that the quotient $X'/\tau$ must be an Enriques surface, and by \cite[Thm~3.2]{MLP}, this is the unique involution on any Godeaux surface such that the quotient is an Enriques surface. We work out this quotient in detail below.
\end{rem}

First note that $H^0(3K_{Y'})^{\text{inv}}=\left<x_1^2x_2,x_2x_3^2\right>$, and so $|3K_{Y'}^{\text{inv}}-R_0|$ is the pencil in $|2K_{Y'}|$ spanned by $x_1^2$, $x_3^2$; cf.~\cite[\S4]{KL}. Next consider the pencil $\Lambda$ in $|K_{Y'}|$ spanned by $x_1,x_3$, of genus $5$ curves on $Y'$ with four base points forming the $\sigma$-orbit $Q_1$. By construction of $Y'$, we see that $\sigma$ acts as an involution on $\Lambda$. Thus each curve in $|3K_{Y'}^{\text{inv}}-R_0|$ is a reducible curve $C+\sigma(C)$ for $C$ in $\Lambda$, with a node at each of the four basepoints.

\begin{lem}
The image $\Lambda_{X'}$ of $\Lambda$ on $X'$ is a pencil of curves of geometric genus $3$ with one node at the only basepoint $P_1$. Each curve in $\Lambda_{X'}$ is the image of a $\sigma$-orbit of two curves in $\Lambda$. The other four isolated fixed points $P_2,\dots,P_5$ of $\tau$ appear as second nodes on four distinct curves in $\Lambda_{X'}$. Finally, there are two $\tau$-invariant curves in $\Lambda_{X'}$.
\end{lem}

\begin{proof} The description of the general curve in $\Lambda_{X'}$ follows from the discussion preceding the Lemma.
Using an easier version of the same computation as in the proof of Proposition \ref{prop!I4-fibre}, we see that for $i=1,\dots,4$, the $\sigma$-orbit $C_i+\sigma(C_i)$ in $\Lambda$ passing through $Q_i$ has a node at each of the four points of $Q_i$. Thus the image of this orbit under quotient by $\sigma$ is a curve with two nodes, one at $P_1$ and the other at $P_i$. The two $\tau$-invariant curves in $\Lambda$ are the images of $F_1\colon(x_1=0)$ and $F_3\colon(x_3=0)$.
\end{proof}

Let $R$ denote the image of $R_0$ in $X'$. Then $R$ is a smooth genus $2$ curve with $R^2=1$, and the above Lemma shows that $\Lambda_{X'}=|3K_{X'}-R|$.

\begin{cor}
The pencil $\Lambda_{X'}$ descends to an elliptic half pencil on the Enriques surface given by the minimal resolution of the five $A_1$ points in $Z'=X'/\tau$. The image of $R$ is a 2-section of the pencil on $Z'$.
\end{cor}

%{\tiny \marginnote{G: Please add the proof and statement for (I think it is relevant): ``The four other fixed points of $\tau$ appear on four distinct fibers of $\Lambda_{X'}$ and (presumably) these four fibers are $E_1+E_2$."}}

%\begin{lem}
%The action of $\sigma$ is an involution on each fibre, and an involution on the base of $\Lambda$. The image $\Lambda_{X'}$ of $\Lambda$ on $X'$ is a pencil of genus $2$ curves with one basepoint $P_1$.
%The linear system $\Lambda_{X'}=|3K_{X'}-R|$ is a genus two pencil with one base point $P_1$. %Each fibre of $\Lambda_X$ is the image of a $\sigma$-orbit of two curves in $\Lambda$.
%The four other fixed points of $\tau$ appear on four distinct fibres of $\Lambda_X$ and (presumably) these four fibres are $E_1+E_2$. Two fibres are $\tau$-invariant. The pencil $\Lambda_X$ descends to an elliptic half pencil on the Enriques surface $X/\tau$.
%\end{lem}

%{\large $\spadesuit:$ How do we see the other half pencils on $X/\tau$?}

\begin{cor}\label{cor!Z4-smooth-K3-quotient}
For general $Y'$ in $\cM^s$, the quotient $V':=Y'/\tau$ is a nodal K3 surface $(2,2,2)\subset\PP^5$ with four $A_1$ singularities. The minimal resolution of $V'$ is a K3 surface with the elliptic fibration image of the pencil $\Lambda=\left<x_1,x_3\right>\subset|K_{Y'}|$, and so the four $(-2)$-curves coming from the $A_1$ are sections.
\end{cor}

\begin{proof}
The invariant monomials of the $\tau$-action on $Y'$ are $x_2$ , $v_1=x_1^2$, $v_2=x_1x_3$, $v_3=x_3^2$, $y_1$, $y_3$. There is one monomial relation between these generators: $v_1v_3=v_2^2$ and the two quartics $q_0$ and $q_2$ can also be expressed in terms of the invariants. This gives a complete intersection $(4,4,4)\subset\PP(1,2,2,2,2,2)$ which is however, not \emph{well formed}. The variable $x_2$ appears only as a square, so we can divide all weights by two and set $v_0=x_2^2$ to get the K3 surface $(2,2,2)\subset\PP^5$.

By Lemma \ref{lem!Z4-fix-tau}, the $\sigma$-orbit $Q_1$ descends to four $A_1$ singularities on $Y/\tau$, while the other orbits $Q_2,\dots,Q_5$ descend to pairs of nonsingular points on $Y/\tau$. After blowing up the four $A_1$ singularities, the image of $\Lambda$ is an elliptic fibration with ramification curve comprising the 4-section $R_0$ together with the four $(-2)$-curves.
\end{proof}

As we did for $\Z/3$-Godeaux, in the right column of Figure \ref{f3} we show the situation described in this subsection. In the next subsections we will describe and use the degenerate situation shown at the left side of Figure \ref{f3}.

\begin{figure}[htbp]
\centering
\includegraphics[width=17cm]{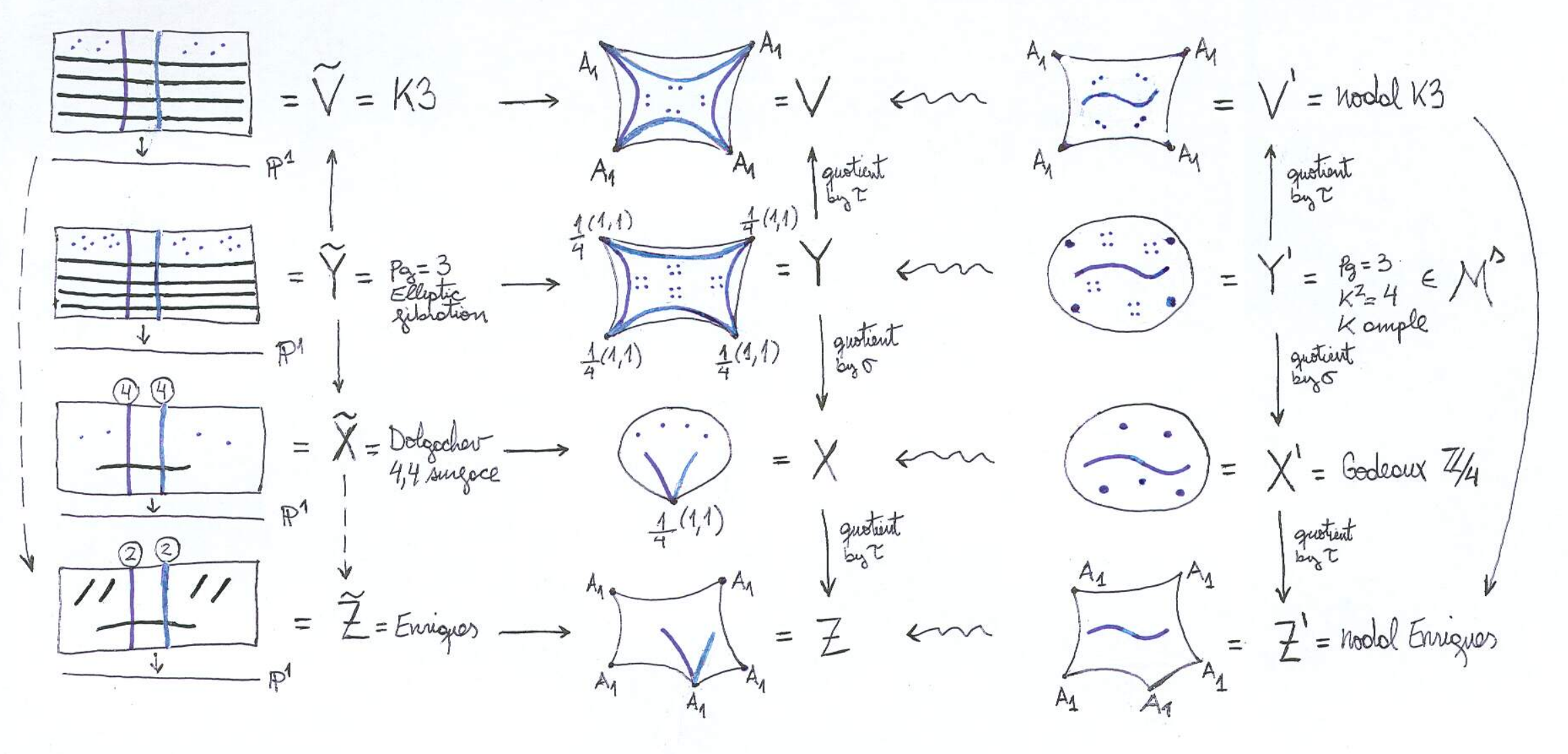}
\caption{Big picture for Section \ref{s3}}
\label{f4}
\end{figure}

%-------------------------------------------------------------------------------------
\subsection{A stable $\ZZ/4$-Godeaux with $\frac{1}{4}(1,1)$ singularity}

We exhibit a divisor $\cY\subset\cM$ in the moduli space corresponding to $\ZZ/4$-Godeaux surfaces with a single $\frac14(1,1)$ singularity, cf.~\cite{K}. Take $Y_{8,8}\subset\PP(1,1,4,4,4)$ with $\ZZ/4$-action
\begin{equation*}
\sigma\colon (u_1,u_3,y_1,y_2,y_3)\mapsto(\eps u_1,\eps^3u_3,iy_1,y_2,-iy_3),
\end{equation*}
where $\eps=\exp(\frac{2\pi i}{8})$. Although $\eps$ is a primitive 8th root of unity, $\sigma$ has order four because of the $\C^*$-action on the weighted projective space. The defining equations $q_0$, $q_2$ of $Y$ are in eigenspaces $0$ and $2$ respectively, and are combinations of the following monomials:
\begin{align*}
q_0 &\colon u_1^8,\ u_1^4u_3^4,\ u_3^8,\ u_1^3u_3y_1,\ u_1^2u_3^2y_2,\ u_1u_3^3y_3,\ y_2^2,\ y_1y_3, \\
q_2 &\colon u_1^6u_3^2,\ u_1^2u_3^6,\ u_1u_3^3y_1,\ u_1^4y_2,\ u_3^4y_2,\ u_1^3u_3y_3,\ y_1^2,\ y_3^2.
\end{align*}

%{\tiny \marginnote{G: Please add some sentence explaining why this is in the same family of Godeaux described before.}}

\begin{lem}\label{lem!Z4-stable-Y} The general member $Y$ in $\cY$ is a complete intersection
\[Y_{2,4,4}\subset\PP(1,1,1,2,2,2)\]
defined by equations $x_1x_3-x_2^2=0$ and $q_i(x_1,x_2,x_3,y_1,y_2,y_3)=0$ for $i=0,2$ where $x_1=u_1^2,x_2=u_1u_3,x_3=u_3^2$ and the $q_i$ are quartics. The four $\frac14(1,1)$ singularities $Y\cap(x_1=x_2=x_3=0)$ form a $\sigma$-orbit. The quotient $X=Y/\sigma$ is a stable $\ZZ/4$-Godeaux surface with a single $\frac14(1,1)$ singularity, and there is a $\Q$-Gorenstein smoothing $X\rightsquigarrow X'$ to a smooth $\ZZ/4$-Godeaux surface.
\end{lem}

\begin{proof}
The general $Y_{8,8}\subset\PP(1,1,4,4,4)$ is quasismooth with four $\frac14(1,1)$ singularities at $Y\cap(u_1=u_3=0)$. Using adjunction, $\omega_Y=\cO_Y(2)$, and we compute $p_g(Y)=3$, $K_Y^2=4$. The graded ring $R(Y,K_Y)=\bigoplus_{k\ge0} H^0(\cO_Y(2k))$ is generated by $x_1=u_1^2$, $x_2=u_1u_3$, $x_2=u_3^2$ in degree 1 and $y_1,y_2,y_3$ in degree 2, and the surface described in the statement of the Lemma is precisely that given by $\Proj R$. The $\sigma$-action extends that of Section \ref{sec!Z4-smooth} with $y_2$ invariant, and this is still fixed point free for general $Y$. The $\Q$-Gorenstein smoothing of the four singularities can be done simultaneously and $\ZZ/4$-equivariantly by varying the quadric equation to $x_1x_3-x_2^2=\lambda y_2$ to eliminate $y_2$ for $\lambda\ne0$.
\end{proof}

\begin{rem} We use the above Lemma to interchange between $Y_{8,8}\subset\PP(1,1,4,4,4)$ and $Y_{2,4,4}\subset\PP(1,1,1,2,2,2)$ without further comment.
\end{rem}
%{\tiny \marginnote{G: Please add more details in the proof.}}

As before, there is a family $\cY^s=\cM^s\cap\cY$ of stable $\ZZ/4$-Godeaux surfaces with an involution. The action of $\tau$ on the ambient space $\PP(1,1,4,4,4)$ is
\[\tau\colon(u_1,u_3,y_1,y_2,y_3)\mapsto(iu_1,-iu_3,y_1,y_2,y_3),\]
and the defining equations of $Y$ are linear combinations of the following monomials
\begin{align*}
q_0 &\colon u_1^8,\ u_1^4u_3^4,\ u_3^8,\ u_1^2u_3^2y_2,\ y_2^2,\ y_1y_3, \\
q_2 &\colon u_1^6u_3^2,\ u_1^2u_3^6,\ u_1^4y_2,\ u_3^4y_2,\ y_1^2,\ y_3^2.
\end{align*}
Again, we can check that $\tau$ extends the smooth case of Section \ref{sec!Z4-smooth}.

%{\tiny \marginnote{G: I am not checking anything that involves polynomials, so please recheck for me everything :)}}

\begin{prop}\label{prop!Z4-stable-involution-fixed}
Let $X$ be a stable Godeaux surface with a single $\frac14(1,1)$ point $P$ and involution $\tau$, arising as the quotient $X=Y/\sigma$ for general $Y$ in $\cY^s$. The fixed curve of $\tau$ is $R=C_1+C_2$, where the intersection $C_1\cap C_2$ is the singular point $P$. There are four isolated fixed points of $\tau$ on $X$.
\end{prop}

\begin{proof}
This is similar to Lemma \ref{lem!Z4-fix-tau}. The fixed curve on $Y$ is defined by $(x_2=0)$, and since $x_2=u_1u_3$, $R$ splits into two components $C_1+C_2$ on $X$. Moreover, the locus $Y\cap(u_1=u_3=0)$ is the $\sigma$-orbit $Q_1$ of $\frac14(1,1)$ singularities, so the $C_i$ intersect at a single $\frac14(1,1)$ point on $X$. The isolated fixed points of $\tau$ are the images of four $\sigma$-orbits $Q_2,\dots,Q_5$ defined by $(y_1=y_3=0)$ on $Y$.
\end{proof}

Let $\phi_Y \colon \Ytilde \to Y$ be the minimal resolution of the four $\frac14(1,1)$ singularities on $Y$. Then $K_{\Ytilde}=\phi_Y^*(K_Y)-\frac12(E_1+\dots+E_4)$, where the $E_i$ are the $(-4)$-exceptional curves. Since $\Ytilde$ is an elliptic surface with $p_g=3$, by the Kodaira bundle formula, we also have $K_{\Ytilde} \sim 2F$ where $F$ is a fibre. The curves $E_i$ are therefore sections of the fibration.

Thus the elliptic fibration $\Ytilde\to\PP^1$ is the resolution of the pencil $\left<u_1,u_3\right>$, and all fibres are complete intersection curves $F_{8,8} \subset \PP(1,4,4,4)$. Unlike in the $\ZZ/3$ case, these fibres are really just hyperplane sections of $Y$, and because $K_{\Ytilde}=2F$, they are nonsingular at the $\frac14(1,1)$ points.

\begin{lem} The elliptic surface $\Ytilde$ associated to a general surface $Y$ in $\cY$ has two nonsingular fibres $F_1$, $F_3$ that are invariant under the action of $\sigma$ and 48 $I_1$ fibres. If $Y$ is in $\cY^s$, then 16 of the nodes in $I_1$ fibres form four $\sigma$-orbits of isolated $\tau$-fixed points.
\end{lem}
\begin{proof}
The automorphism $\sigma$ acts on the base of the fibration only, and the two fibres $F_1\colon(u_1=0)$ and $F_3\colon(u_3=0)$ are invariant. A direct computation shows that for general $Y$, $F_i$ are nonsingular, and the computation of singular fibres and fixed points is again an easier version of the argument in the last paragraph of the proof of Proposition \ref{prop!I4-fibre}.
\end{proof}

Let $\Xtilde$ be the elliptic surface obtained as the quotient $\Ytilde/\sigma$.

%{\tiny \marginnote{G: Please analyze: ``We should be able to analyse the singular fibres of $\Ytilde$ quite easily in this case, and I believe that nothing particularly exciting happens."}}

%\bigskip
%{\large $\spadesuit:$ The discriminant of the $\ZZ/4$ elliptic fibration. Each fibre determines a pencil of quadrics. In general, this pencil is generated by two quadrics of rank 3. A singular fibre is obtained when the generic quadric drops rank to two. This should be sufficient to find the singular fibres. I will write this properly later}
%\bigskip

\begin{cor}
The minimal resolution $\Xtilde$ of $X$ is a $(4,4)$-Dolgachev surface, and the $(-4)$-curve $E$ is a $4$-section. The involution $\tau$ fixes the two multiple fibres pointwise, and $E$ is invariant under $\tau$, with two fixed points at the intersection with the multiple fibres. The quotient $\Xtilde/\tau$ is a nodal Enriques surface with a $(-2)$-curve which is a $2$-section of the elliptic fibration, and four $A_1$-singularities lying on distinct fibres.
\end{cor}

%{\tiny \marginnote{G: Why $A_1$ on distinct fibers? Has to do with before.}}

\begin{proof}
The two $\sigma$-invariant fibres $F_i$ on $\Ytilde$ give rise to two fibres of multiplicity $4$ on the quotient $\Xtilde$. From Proposition \ref{prop!Z4-stable-involution-fixed}, these two fibres are pointwise fixed by $\tau$, and so their multiplicity on $X/\tau$ drops to two. The image of $E$ is a $(-2)$-curve because $E$ is $\tau$-invariant.
\end{proof}

%{\tiny \marginnote{G: Please add some reason to irreducible nonsingular fixed fibers.}}

\begin{cor}
For general $Y$ in $\cY^s$ with a $\frac{1}{4}(1,1)$ singularity, the quotient surface $Y/\tau$ is a nodal K3 surface with four $A_1$ singularities. The ramification curve breaks into two components, each of which passes through all four $A_1$ singularities.
\end{cor}

\begin{proof}
%{\tiny \marginnote{G: Please improve the exposition in the three last paragraphs of the proof.}}
The fact that the quotient is an intersection of three quadrics is similar to Corollary \ref{cor!Z4-smooth-K3-quotient}, starting from the description of $Y$ in Lemma \ref{lem!Z4-stable-Y}.

The singularities are again induced by the $\sigma$-orbit $Q_1$ on $Y$, but this time $Q_1$ comprises four $\frac14(1,1)$ singularities. The local orbifold coordinates near a $\frac14(1,1)$ point $P$ of $Q_1$ are $u_1$, $u_3$, and the action of $\tau$ on these coordinates is $\frac14(1,3)$. Thus the invariant monomials of the composite action are $w_1=u_1^4$, $w_2=u_1^2u_3^2$, and $w_3=u_3^4$ with single relation $w_1w_3=w_2^2$, and the image of $P$ on $Y/\tau$ is an $A_1$ singularity.
\end{proof}
%Claim: These remain $A_1$ singularities.

%Rough Proof: Consider $\C^2/\frac14(1,1)$, then by \eqref{eqn!Z4-tau}, the action of $\tau$ lifts to $\C^2$ as $\frac18(1,3)$. Working modulo $\frac14(1,1)$, this means that $\tau=\frac18(0,4)=\frac12(0,1)$. An explicit computation of invariants then shows that the quotient reduces to an $A_1$. Probably a more conceptual proof of this, is to resolve the $\frac14(1,1)$ and show that the action of $\tau$ on the resolution gives the resolution of an $A_1$ sing.

%--------------------------------------------------------------------------------------------
\subsection{Stable rational $\ZZ/4$-Godeaux surfaces}

Following the strategy explained in Section \ref{s23}, we look for stable $\ZZ/4$-Godeaux surfaces $\Xtilde$ with an elliptic fibration where the multiple fibres become nodal. As in Section \ref{s23}, this behaviour is induced by $Y$ having a single $\sigma$-fixed point in the hyperplane $(x_1=0)$ (or $(x_3=0)$). Such a fixed point must be an $A_3$ singularity or worse. This condition is independent of the presence of $\frac14(1,1)$ singularities.

\begin{lem}
The following divisors $\cB_1,\cB_3\subset\cM$ defined by
\begin{align*}
\cB_1 &\colon \text{no }x_1^4\text{ monomial in }q_0 \\
\cB_3 &\colon \text{no }x_3^4\text{ monomial in }q_0
\end{align*}
parametrise surfaces $Y$ with an $A_3$ singularity at $P_1=(1,0,0,0,0)$ (respectively $P_3=(0,0,1,0,0)$) which is fixed under the action of $\sigma$. Moreover, the quotient $X=Y/\sigma$ of a general surface in $\cB_i$ is a Godeaux surface with a $\frac 1{16}(1,3)$ singularity.
\end{lem}

\begin{proof}
Consider a general surface $Y$ in $\cB_1$. Since $q_0$ does not contain the monomial $x_1^4$, the coordinate point $P=P_{x_1}$ is contained in $Y$. Ignoring coefficients, the leading terms of $q_0$, $q_2$ at $P$ are
\[y_1y_3+x_2y_1+x_3^2+h.o.t., \ x_3+x_2^2+h.o.t.\]
Thus eliminating $x_3$ with $x_2^2$ and a coordinates change shows that $Y$ has an $A_3$ singularity at $P$, and clearly $P$ is fixed under the action of $\sigma$. Computing with local coordinates near $P$, we see that the action of $\sigma$ on $(y_1,y_3,x_2)$ is $(-iy_1,iy_3,ix_2)$ and so $P$ induces a $\frac 1{16}(1,3)$ point on $X$.
\end{proof}

%{\tiny \marginnote{G: Please add details about intersection and independency. It justifies the picture :)}}

\begin{lem}\label{lem!Z4-transversal} All intersections between $\cY$, $\cB_1$ and $\cB_3$ are transversal.
\end{lem}
\begin{proof} The proof is similar to that of Lemma \ref{lem!Z3-transversal}, and we do not give full details. We only remark that a surface $Y$ in $\cY\cap\cB_1$ is realised as a complete intersection $Y_{8,8}\subset\PP(1,1,4,4,4)$ for which $q_0$ does not contain the monomial $u_1^8$. The $\Q$-Gorenstein smoothing of the $\frac14(1,1)$ point is in Lemma \ref{lem!Z4-stable-Y}, and the $\frac1{16}(1,3)$ point is smoothed by allowing $u_1^8$ to appear again.
\end{proof}

In this case, the $A_3$ singularities on $Y$ are fixed under $\tau$, so it is not true that $\cB_i^s=(\cB_1\cap\cB_3)^s$ here, unlike Section \ref{s23}.

%We can also impose an involution when we want. (This is nontrivial because we should check that the tangent cone still gives an $A_3$ even when there are less monomials available). {\tiny \marginnote{G: Please make the involution comment more precise or erase.}}

Let $X_1$ be the stable Godeaux surface with one $\frac{1}{16}(1,3)$, $X_2$ with singularities $\frac{1}{4}(1,1)$ and $\frac{1}{16}(1,3)$, and $X_3$ with one $\frac{1}{4}(1,1)$ and two $\frac{1}{16}(1,3)$ singularities. Let $Y_i$ be the corresponding cyclic cover of $X_i$. According to Lemma \ref{lem!Z4-transversal}, the surfaces are related under $\Q$-Gorenstein degenerations as shown in Figure \ref{f5}.

\begin{figure}[htbp]
\centering
\includegraphics[width=14cm]{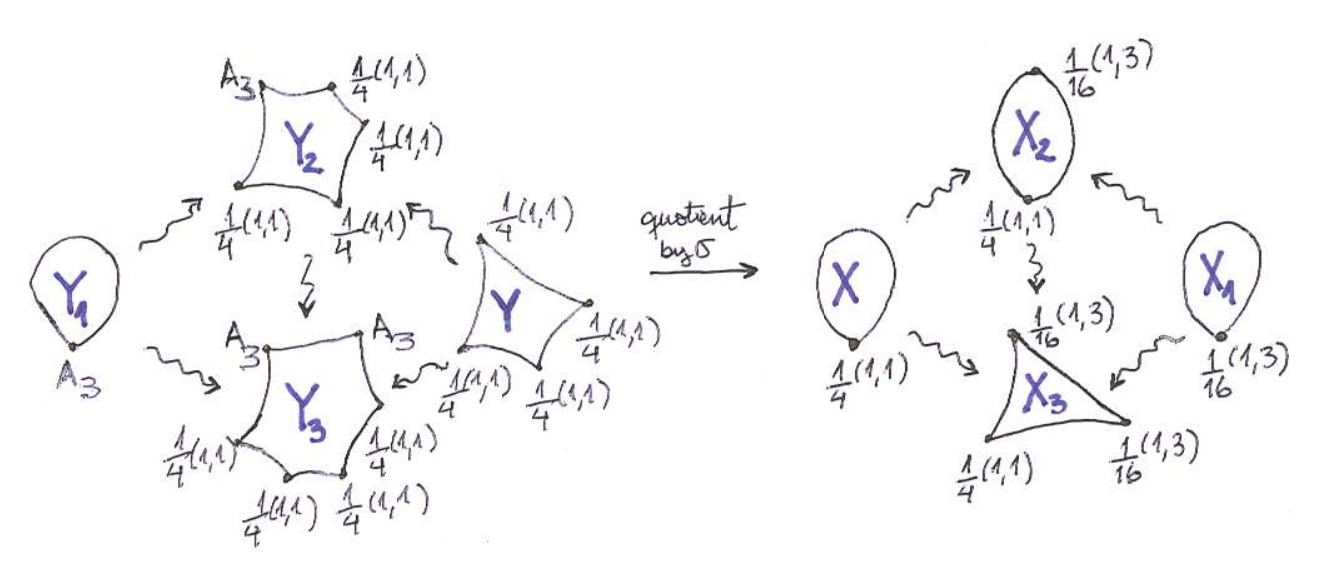}
\caption{Stable $\Z/4$-degenerations of $Y$ and $X$}
\label{f5}
\end{figure}

\begin{prop}
The surfaces $X_i$ are rational for $i=2,3$.
\end{prop}

\begin{proof}
This is the same proof as in the $\Z/3$ case, Proposition \ref{prop!rational-X23}.

%\bigskip
%{\large $\spadesuit:$ To finish the proof, I need a more explicit model for $X_3$. I may get it if you can give me some other curve in the surface. For example, it would be extremely useful to have sections for the elliptic fibration $\widetilde X_3$, maybe it is not difficult to get them...}
%\bigskip

\end{proof}

\begin{prop}
The surface $X_1$ is rational.
\end{prop}

\begin{proof}
Let us consider a $\Q$-Gorenstein deformation $X_1 \rightsquigarrow X_2$ over a disk $\D$, so that the deformation around the singularity $\frac{1}{16}(1,3)$ is constant. Then we resolve it simultaneously (and minimally) to obtain a $\Q$-Gorenstein smoothing $X'_1 \rightsquigarrow X'_2$. By the previous proposition, in the central fibre $X'_2$ we have a $(-1)$-curve between the $(-6)$-curve and the last $(-2)$-curve of the exceptional divisor. This $(-1)$-curve intersects the $(-4)$-curve of the exceptional divisor of $\widetilde X_2 \to X_2$ at one point and transversally. Therefore, we can apply the flip used in Proposition \ref{flip}. After the flip, the deformation $X'_1 \rightsquigarrow X'_2$ becomes smooth, and so the Kodaira dimension is preserved. By the previous proposition, the surface $X_2$ is rational, and so $X_1$ is rational as well.
\end{proof}

%\bigskip
%{\large $\spadesuit:$ I would like to give a $\P^2$ model for the one with 3 singularities. For that I need more curves, maybe just one more. I already have possibilities. After that, I would be able to degenerate further.}
%\bigskip

%\bigskip
%{\large $\spadesuit:$ The situation of $\Ytilde \to \Xtilde$ on the fixed fibre is locally $A_{3d-1} \to \frac{1}{3d}(1,3d-1)$. You have found $d=1$. Possibilities in our case are $d=1,2,3,4,5,6,7$. Can you find $A_2,\ldots,A_{20}$? This is not too important, but maybe it is possible. This is like finding a $\Ytilde$ with an $I_{21}$ fibre. Maybe you can do it using discriminant if available.}
%\bigskip

%--------------------------------------------------------------------------------------------
%\subsection{Other involutions on the $\ZZ/4$-Godeaux}
\begin{rem}
%{\tiny \marginnote{G: Let us remark at least what the quotients are, or we could erase this subsection, up to you.}}
We know of one other involution on the $\ZZ/4$-Godeaux surface, from \cite[\S4.2]{KL}. The automorphism $\tau\colon (x_1,x_2,x_3,y_1,y_3)\mapsto(x_3,x_2,x_1,y_3,y_1)$
acts on a subfamily of $\cM$, and the group generated by $\sigma$ and $\tau$ is $D_4$. The fixed curve is genus $1$ and there are five isolated fixed points. According to \cite[\S8.7]{CCML}, the quotient $X/\tau$ is a double plane of du Val type. It would be interesting to study the corresponding degenerations for this involution, and understand the behaviour of the double plane.
%\item Second involution is
%\[(x_1,x_2,x_3,y_1,y_3)\mapsto(-x_1,x_2,x_3,y_1,-y_3)\]
%and the big group is now $\ZZ/2\times\ZZ/4$. (NOTE I THINK THIS ONE DOES NOT WORK BECAUSE IT FORCES THE SURFACE TO BE SINGULAR)
\end{rem} 
%----------------------------------------------------------------------------------

%--------------------------------------------------------------------------------------------

\Addresses

\end{document}